\documentclass[a4paper,10pt,twoside]{amsart}
\usepackage[left=2.7cm,right=2.7cm,top=3.5cm,bottom=3cm]{geometry}
\usepackage[english]{babel}
\usepackage[latin1]{inputenc}
\usepackage{amsmath}
\usepackage{amsfonts}
\usepackage{amsthm}
\usepackage{amscd}
\usepackage{amssymb}
\usepackage[all]{xy}
\usepackage{graphicx}
\usepackage[usenames,dvipsnames]{color}

\theoremstyle{plain}
\newtheorem{thm}{Theorem}[section]
\newtheorem{cor}[thm]{Corollary}
\newtheorem{lem}[thm]{Lemma}
\newtheorem{prop}[thm]{Proposition}

\theoremstyle{definition}
\newtheorem{dfn}[thm]{Definition}
\newtheorem{eg}[thm]{Example}
\newtheorem{rmk}[thm]{Remark}
\newtheorem{rmks}[thm]{Remarks}
\newtheorem{ass}[thm]{Assumption}
\newenvironment{prf}{\begin{proof}[Proof]}{\end{proof}}
\newenvironment{skprf}{\begin{proof}[Sketch of proof]}{\end{proof}}

\renewcommand{\ge}{\geqslant}
\renewcommand{\le}{\leqslant}

\newcommand{\field}[1]{\mathbb{#1}}
\newcommand{\Q}{\field{Q}}
\newcommand{\C}{\field{C}}
\newcommand{\R}{\field{R}}
\newcommand{\N}{\field{N}}
\newcommand{\Z}{\field{Z}}
\newcommand{\A}{\field{A}}
\newcommand{\F}{\field{F}}
\newcommand{\p}{\field{P}}

\newcommand{\idl}{{\bf I}}

\newcommand{\za}{\widehat{\Z}}
\newcommand{\da}{\widehat D}
\newcommand{\xa}{\widehat X}
\newcommand{\dap}{\widehat D_\gotp}
\newcommand{\wds}{\widehat{D^*}}
\newcommand{\hpi}{\widehat\pi}
\newcommand{\cpr}{\widehat{\calp}}

\DeclareMathOperator{\GL}{GL}
\DeclareMathOperator{\SL}{SL}

\DeclareMathOperator{\dlog}{dlog}
\DeclareMathOperator{\supp}{supp}

\DeclareMathOperator{\Irr}{Irr}

\DeclareMathOperator{\car}{char}
\DeclareMathOperator{\cls}{Cl}

\DeclareMathOperator{\dlib}{d}
\newcommand{\dsbk}{\dlib_{\rm Bk}}
\newcommand{\dsbkp}{\dsbk^+}
\newcommand{\dsbkm}{\dsbk^-}

\newcommand{\dsas}{\dlib_{\rm as}}
\newcommand{\dsbas}{\dlib_{B\!,{\rm as}}}
\newcommand{\dsbasp}{\dsbas^+}
\newcommand{\dsbasm}{\dsbas^-}

\newcommand{\dsan}{\dlib_{\rm an}}
\newcommand{\dsanp}{\dsan^+}
\newcommand{\dsanm}{\dsan^-}
\newcommand{\dsans}{\dlib_{\rm an*}}

\newcommand{\dsl}{\dlib_{\log}}

\newcommand{\dsasp}{\dsas^+}  
\newcommand{\dsasm}{\dsas^-}

\newcommand{\cala}{\mathcal A}

\newcommand{\cald}{\mathcal D}

\newcommand{\cali}{\mathcal I}
\newcommand{\calj}{\mathcal J}
\newcommand{\calk}{\mathcal K}

\newcommand{\calo}{\mathcal O}
\newcommand{\calp}{\mathcal P}
\newcommand{\calq}{\mathcal Q}

\newcommand{\cals}{\mathcal S}
\newcommand{\calt}{\mathcal T}
\newcommand{\calu}{\mathcal U}
\newcommand{\calv}{\mathcal V}

\newcommand{\calx}{\mathcal X}
\newcommand{\caly}{\mathcal Y}
\newcommand{\calz}{\mathcal Z}
\newcommand{\gota}{\mathfrak a}
\newcommand{\gotb}{\mathfrak b}
\newcommand{\gotc}{\mathfrak c}
\newcommand{\gotd}{\mathfrak d}

\newcommand{\gotk}{\mathfrak k}
\newcommand{\gotm}{\mathfrak m}
\newcommand{\gotn}{\mathfrak n}
\newcommand\gotp{\mathfrak p}
\newcommand\gotq{\mathfrak q}
\newcommand{\gotr}{\mathfrak r}

\newcommand{\buno}{\mathbf 1}

\newcommand{\hooklongrightarrow}{\lhook\joinrel\longrightarrow}

\newcommand{\interior}[1]{%
{\kern0pt#1}^{\mathrm{o}}%
}

\newfont{\cyr}{wncyr10 scaled 1000}

\begin{document}

\title[Densities on Dedekind domains, completions and Haar measure]{Densities on Dedekind domains, completions and Haar measure}

\author[Demangos]{Luca Demangos}
\address{Department of Mathematics, MB423B, Xi'an Jiaotong - Liverpool University, 215123 Suzhou, China}
\email{Luca.Demangos@xjtlu.edu.cn} 

\author[Longhi]{Ignazio Longhi}
\address{Department of Mathematics, Indian Institute of Science, Bangalore.\newline
Current affiliation: Dipartimento di Matematica dell'Universit\`a di Torino, via Carlo Alberto 10, 10123, Torino, Italy}
\email{ignazio.longhi@unito.it}

\begin{abstract}  Let $D$ be the ring of $S$-integers in a global field and $\da$ its profinite completion. Given $X\subseteq D^n$, we consider its closure $\xa\subseteq\da^n$ and ask what can be learned from $\xa$ about the ``size'' of $X$. In particular, we ask when the density of $X$ is equal to the Haar measure of $\xa$. 

We provide a general definition of density which encompasses the most commonly used ones. Using it we give a necessary and sufficient condition for the equality between density and measure which subsumes a criterion due to Poonen and Stoll. We also show how Ekedahl's sieve fits into our setting and find conditions ensuring that $\xa$ can be written as a product of local closures.

In another direction, we extend the Davenport-Erd\H os theorem to every $D$ as above and offer a new interpretation of it as a ``density=measure'' result. Our point of view also provides a simple proof that in any $D$ the set of elements divisible by at most $k$ distinct primes has density 0 for any $k\in\N$.

Finally, we show that the closure of the set of prime elements of $D$ is the union of the group of units of $\da$ with a negligible part. \end{abstract}

\maketitle
\tableofcontents

\section{Introduction}

\subsection{Motivation} One of the most basic questions in number theory is to find out how big are certain interesting subsets of $\N$, the set of natural numbers. 

It is obvious that finite sets must be very small; moreover, one can give a quantitative assessment of the size of a finite $X\subset\N$ just by computing its cardinality. For an infinite $X$ the traditional approach is to use some form of density, most often the asymptotic one
$$\dsas(X):=\lim_{r\rightarrow\infty}\frac{|X\cap[0,r]|}{r}\,,$$
but many other definitions of density have found their use in number theory. Intuitively, a density should behave similarly to a measure (see for example the discussion in \cite[III.1]{tnm}) and, starting at least with \cite{buck}, much research has gone into fitting densities into a measure theory framework (we mention \cite{fga}, as well as the notion of density measure developed in papers like \cite{mhr} and \cite{vndw}).

One can also observe that densities are usually defined exploiting the canonical injection $\N\hookrightarrow\R$ and that $\R$ is not the only topological space with nice properties into which one could embed $\N$. Number theorists know that $\R$ is just one of many completions of $\Q$; and the principle that ``all places were created with equal rights'' suggests of looking instead at the embedding of $\N$ into $\za\simeq\prod\Z_p$, the profinite completion of $\Z$. As a compact group, $\za$ is endowed with a Haar measure $\mu$ and it is easy to see that if $X$ is an arithmetic progression then its asymptotic density is exactly $\mu(\xa)$, where $\xa$ is the closure of $X$ in $\za$. This fact has been noticed by many people and the idea of evaluating the size of a general $X\subseteq\Z$ by computing $\mu(\xa)$ is natural enough that it was independently developed by several mathematicians. It was already implicitly behind Buck's paper \cite{buck} (even if, as far as we can judge, most likely Buck did not know about $\za$); and it arose again in works by Novoselov \cite{nov}, Ekedahl \cite{ekd}, Poonen and Stoll \cite{ps1}, Kubota and Sugita \cite{ks1} - apparently all unaware of their predecessors (the partial exception is \cite{ps1}, whose authors learned of \cite{ekd} just after completing their own paper). In particular, \cite{nov} and \cite{ks1} were motivated by the fact that $\za$ can be thought of as a probability space: we recommend \cite{ind1} and \cite{ind2} for a survey of different approaches to probabilistic number theory by compactifications of $\N$.

Introducing $\za$ makes it quite straightforward to extend the same approach to the setting of number fields or, more generally, global fields (i.e., finite extensions of either $\Q$ or 
$\F_p(t)$\,). Throughout this paper the expression {\em global Dedekind domain} will be used as a shortening for the ring of $S$-integers of a global field: that  is, $D$ is a global Dedekind domain if its fraction field $F$ is a global field and there is a non-empty finite set $S$ of places of $F$ (containing the archimedean ones in the characteristic $0$ case) such that
$$D=\{x\in F\mid v_\nu(x)\ge0\;\forall\,\nu\notin S\}\,,$$
where $v_\nu$ is the valuation attached to $\nu$.

For the rest of this introduction, let $D$ be a global Dedekind domain. Then $D$ can be embedded in its profinite ring completion $\da$ - a compact topological ring endowed with a Haar measure $\mu$ (the construction of $\da$ and $\mu$ will be recalled in Sections \ref{s:prelim} and \ref{s:m&d} respectively). Moreover, the classical notion of asymptotic density of subsets of $\N$ can be extended (in many different ways) to subsets of $D^n$. So, letting $\xa$ denote the closure in $\da^n$ of $X\subseteq D^n$, one can ask about the comparison between $\mu(\xa)$ and the value at $X$ of a given density on $D^n$. More generally, one might ask what can be learned about $X$ by looking at $\xa$. Giving some answers to these questions is the main goal of this paper.

\subsection{Our results} Assume that $X$ is closed with respect to the topology induced by the embedding $D\hookrightarrow\da$. Then $X=D^n\cap\xa$, so that $X$ and $\xa$ are as close as 
possible. Our basic philosophy is that in this case finding out $\xa$ is the best way of assessing how large is $X$.

This is trivial when $X$ is finite (since then $X=\xa$). If $X$ is infinite, a first natural question is whether one has 
\begin{equation} \label{e:dmintr} \mu(\xa)=d(X) \end{equation} 
for some given density $d$. As mentioned before, there are many different possible definitions of density on $D^n$, so in \S\ref{sss:dns} we list some hypotheses that $d$ should satisfy for the comparison to make sense. 

Equality \eqref{e:dmintr} holds trivially when $\xa$ is compact and open (Lemma \ref{l:d=m0}) or, with little more work, when its boundary has measure $0$ (Lemma \ref{l:brtt}). These cases 
can be treated by Buck's measure theory, as done in \cite{buck} and the papers inspired by it, like \cite{lonied} or \cite{fptw} (we shall briefly summarise the main ideas in \S\ref{sss:Bd}; see also Remark \ref{r:bmsr} and Corollary \ref{c:bmsr}). On the other hand, we shall see that the approach of \cite{buck} cannot be applied to evaluate, for example, the size of the set of square-free integers - whose density $\frac{6}{\pi^2}$ is exactly what one would expect from \eqref{e:dmintr}. \footnote{The heuristics for computing the density of square-free integers is based on the fact that in an interval of length $x$ for every prime $p$ there are about $(1-p^{-2})x$ numbers which are not divisible by $p^2$. The decomposition $\za=\prod\Z_p$ (see Theorem \ref{t:crt} for a more precise statement) implies that the measure of square-free elements in $\za$ is $\prod_p(1-p^{-2})$.}\\

We have results in three directions.

\subsubsection{Sets of multiples and the Davenport-Erd\H os theorem} As far as we know, the first proof of a version of \eqref{e:dmintr} covering a vast generality of non-trivial cases 
was obtained by Davenport and Erd\H os in \cite{de1}, where they showed that sets of multiples have a logarithmic density. A set of multiples $U$ is basically a union of ideals, so taking the complement in $\Z$ of this union one gets a closed set $X$ and the computation in \cite{de1} yields exactly \eqref{e:dmintr}.

In Theorem \ref{t:de} we extend this result to any global Dedekind domain $D$. Our proof follows quite closely the one in \cite{de1}, where the main tool is the analytic density $\dsan$\,, which we define in \S\ref{sss:anltid} for subsets not of $D$, but of $\cali(D)$ (the set of non-zero ideals of our ring). In \S\ref{sss:sprn}, we show how to densely embed $\cali(D)$ into a compact topological space $\cals(D)$, whose elements we call supernatural ideals (since they generalize the notion of supernatural numbers). There is a natural surjection $\rho$ of $\da$ onto $\cals(D)$ and the measure $\mu$ can be push-forwarded to a measure $\rho_*\mu$ on $\cals(D)$. As for  the set  to be considered, we take as $U$ a union of open ideals in $\da$ and then we have to distinguish between $Y$ (the complement of $U$ in $\da$) and $X$ (the complement of $U$ in $D$): the examples in Remark \ref{r:XvY} show that in general one should expect $\xa$ to be properly contained in $\hat Y=Y$. The version of \eqref{e:dmintr} that we obtain in Theorem \ref{t:de} is $\rho_*\mu(Y)=\dsan(\rho(Y))$.

When $U$ is a union of pairwise coprime ideals ``with enough divisors'' we can prove $\xa=Y$ (Proposition \ref{p:appirr}) and in Theorem \ref{t:asdmltp} we show that under this hypothesis one has the equality of $\mu(\xa)$ with the asymptotic density of $X$ (defined in \S\ref{sss:asnt}). 

In our discussion of sets of multiples, we also prove that the set of integers divisible by at most $k$ distinct primes has density $0$ for every $k\in\N$. Actually, with our approach it is quite easy to show that the analogous statement holds in any global Dedekind domain and for every density satisfying the conditions in \S\ref{sss:dns} (see Proposition \ref{p:msromg} and Corollary \ref{c:lt2}).

\subsubsection{The local-to-global approach: Eulerian sets} The papers \cite{ekd} and \cite{ps1} brought forward a decisive progress in establishing new cases of \eqref{e:dmintr}. Ekedahl proved that \eqref{e:dmintr} is true if $X$ consists of the points in $\Z^n$ which stay away from an affine closed scheme $Z$ when reduced modulo $p$, as $p$ varies over all primes (\cite[Theorem 1.2]{ekd}). \footnote{In the same paper, there is also the claim that any open subset of $\za^n$ has density equal to its measure (\cite[Proposition 2.2]{ekd}). Unfortunately this statement and its consequence \cite[Theorem 2.3]{ekd} are wrong, as we are going to discuss in \S\ref{sss:cnt}.} Poonen and Stoll in \cite[Lemma 1]{ps1} isolated a sufficient (but maybe not necessary) criterion for \eqref{e:dmintr} to hold; this criterion (proved in \cite[Lemma 20]{ps2}) can also be used to simplify the proof of Ekedahl's theorem. Such results were later improved and extended by many authors: we cite, with no claim of exhaustiveness, \cite{bh}, \cite{bsw}, \cite{bbl}, \cite{mch} and \cite{cs}.

All these works take a local-to-global approach. Unique factorization yields a canonical isomorphism $\za\simeq\prod_p\Z_p$ (and, more generally, $\da\simeq\prod_\gotp\dap$, where $\gotp$ varies among all non-zero prime ideals of $D$ and $\dap$ denotes the completion in the $\gotp$-adic topology). In \cite{ekd}, \cite{ps1} and the literature following them, the idea is to start with a family of measurable subsets $U_\gotp\subseteq\dap^n$ and prove an equality of the form
\begin{equation} \label{e:dmintr2} \prod_\gotp\mu_\gotp(U_\gotp)=d(X)\,, \end{equation}
where $\mu_\gotp$ is the Haar measure on $\dap^n$ and
\begin{equation} \label{e:lclcnd} X=D^n\cap\prod U_\gotp\,. \end{equation}
We will show in Proposition \ref{p:d>mc} how in many important cases \eqref{e:dmintr2} implies 
\begin{equation} \label{e:elr0} \xa=\prod U_\gotp \end{equation}
and hence \eqref{e:dmintr}.

In Theorem \ref{t:mt} we give a necessary and sufficient condition for \eqref{e:dmintr} which refines and extends the Poonen-Stoll criterion (in Proposition \ref{p:ps} we explain how to derive the latter from the former). Our result is valid for every global Dedekind domain and for every density as in \S\ref{sss:dns}. The proof of Theorem \ref{t:mt} is not too different from those of extensions of \cite[Lemma 20]{ps2} already in the literature (like \cite[Lemma 3.1]{bbl}, \cite[Theorem 2.1]{mch} or \cite[Proposition 3.2]{cs}), but it might have some advantages in terms of generality and simplicity. In particular we do not require that $X$ is defined by local conditions at primes $\gotp$.

We say that $X$ is {\em Eulerian} if it satisfies \eqref{e:elr0} (Definition \ref{d:elr}, generalized to sets {\em of Eulerian type} in Definition \ref{d:elrtyp}). This notion can be seen as a somewhat more simple-minded extension of strong approximation: see \S\ref{sss:strapprx} for a brief discussion on how they relate. Eulerian sets are the ones for which it is easiest to compute $\mu(\xa)$ and to our knowledge they cover most of the examples where \eqref{e:dmintr} is known to be true. In Section \ref{s:elr} we discuss some conditions which ensure that a set is Eulerian without passing through \eqref{e:dmintr2}: for example, images of polynomials, tuples with a fixed greatest common divisor and $k$-free numbers form Eulerian sets (Proposition \ref{p:plyneulr}, Corollary \ref{c:cprpr} and Corollary \ref{c:kfrelr} respectively). We conclude the paper with some questions (yet too speculative for the name of conjecture) about what sets could be Eulerian and whether this can entail anything about their density.\\

We would like to note that in our work the emphasis is somehow shifted with respect to previous research: one could say that we take a ``global-to-local'' approach, in the sense that 
our interest is in $\xa$, pertaining to the global object $X$ rather than to the local conditions $\prod U_\gotp$. Even when $X$ is defined by \eqref{e:lclcnd}, there is no need for equality \eqref{e:elr0} to be true (just think of the case $D=\Z$, $U_\gotp=\Z_p^*$ for all $p$) and finding out the difference between its two sides can be of some interest. 

We also note that, letting $X(\gotp)$ be the closure of $X$ in $\dap^n$, one can think of the set-theoretic difference
$$\;\prod X(\gotp)-\xa$$
as a way of measuring the failure of the Hasse principle.

\begin{rmks} \begin{itemize} \item[]
\item[{\bf 1.}] For simplicity, in this paper we only work in the affine setting: that is, we only consider subsets of $D^n$. It would be interesting to extend our approach to subsets of the projective space $\p^n(F)$ (where $F$ is the fraction field of $D$) as in \cite[\S3.2]{bbl}.
\item[{\bf 2.}] Obviously our method fails to see local conditions determined by places of $S$ - for the most trivial example, just note that $\N$ and $\Z$ have the same closure $\za$, but are easy to distinguish via the order relation in $\R$. We don't think of this as a serious shortcoming, in the sense that if one has to add finitely many conditions from places in $S$, it is usually not too hard to dispose of them by other techniques after dealing with the infinitely many coming from primes of $D$. \end{itemize} \end{rmks}

\subsubsection{Sets of measure $0$ and the closure of primes} Equality \eqref{e:dmintr} is always true when $\mu(\xa)=0$ (Corollary \ref{c:m=0}). However, sets of density $0$ can still be quite large and assessing their size in a more precise way can be of great interest: for a typical example it is enough to think of $\calp$, the prime numbers in $\N$. Because of compactness, an infinite $X\subseteq D^n$ has always accumulation points in $\da^n$, so there is hope of learning something from $\xa$ even when $\mu(\xa)=0$.  This hope is confirmed in the case of $\calp$: it turns out that $\cpr$ contains $\za^*$, the group of units of $\za$, and that this inclusion is equivalent to the existence of infinitely many primes in arithmetic progressions. (As far as we know, this observation was first made by Lubotzky - see e.g. \cite[page 476]{lub}.)

We generalize this to any global Dedekind domain $D$ in Theorem \ref{t:chbt}. More precisely, we consider the sets $\Irr(D)$, consisting of those irreducible elements of $D$ which generate a prime ideal, and $\da^*$, the group of units of $\da$ (not to be confused with its subgroup $\wds$, which is the closure of $D^*$ in $\da$). We prove that the closure of $\Irr(D)$ in $\da$ is the disjoint union of $\da^*$ and a much smaller piece (the meaning of ``much smaller'' will be explained in \S\ref{sss:clD*}). As one could expect, the main tool for the proof is a weak version of the Chebotarev density theorem (Theorem \ref{t:rfr}).

Theorem \ref{t:chbt} suggests that the method of checking closures in $\da$ can also be applied to study, for example, regularity in the distribution of primes. We will not pursue the matter further in this paper \footnote{Interested readers can find some more considerations on this in \cite[\S3.3.2]{dlms}. An extensive discussion of this topic is in preparation (joint work with F.M.~Saettone).} and will be content to use this result just as a tool to determine $\xa$ in some other cases (Proposition \ref{p:appirr}). But we want to emphasize that Theorem \ref{t:chbt} gives a convincing argument in favour of our point: if $X$ is a closed subset of $D$ (or, at least, not too far  from being closed), then what one 
can learn about $X$ from $\xa$ goes well beyond equality \eqref{e:dmintr}.

\subsection{Structure of this paper} The contents of Sections \ref{s:clsirr}, \ref{s:anlt} and \ref{s:elr} have already been discussed in enough detail. Here we just want to add a few more words about the rest of the paper, in order to put it better in context and to help readers find their way through it.

In Section \ref{s:prelim} we set down the notation and establish a number of facts about $\da$ which we are going to use later. None of these results is new and all of them are well known to experts, who are advised to read just the parts needed to familiarize with our notations.

Section \ref{s:m&d} is dedicated to the comparison between the Haar measure on $\da$ and a general notion of density on $D$. We introduce both notions from scratch. There is not much to say on the measure and here we just remark that putting together Example \ref{eg:da*0} and Theorem \ref{t:chbt} one gets an interpretation of the value at $1$ of the Dedekind zeta function as the inverse of the Haar measure of the closure of $\Irr(D)$.

As for densities, in the literature one can find a vast number of different definitions over $\N$ - see for example \cite{gr} or \cite[Chapter III.1]{tnm} for gentle introductions. 
For the more general situation we are interested in, the literature is less huge, but still quite big and we make no attempt to survey it, just citing on occasion a few relevant papers. In the case of $\N$, a general definition of density was proposed in \cite{fga} and systems of axioms are discussed in \cite{fs}, \cite{gr} and (quite exhaustively) in \cite{lt}; the latter paper, which also considers the case of $\Z$, fits reasonably well with our needs (see Remark \ref{r:cfltaxm}). In \S\ref{sss:dns} we will give a working definition of a density on $D^n$, but readers should be aware that our goal is simply to isolate the properties it should satisfy for comparison with the Haar measure on $\da^n$. A number of examples of such densities will be discussed in full detail in the companion paper \cite{dl3}; in the present work, \S\ref{sss:Bd}-\S\ref{sss:anlt}, we just quickly touch upon some cases that will be needed later.

Finally, we point out that often we work with a Dedekind domain $D$ which is not necessarily global. We always ask that $D$ is countable and $\da$ is compact (Assumption \ref{a:cntb}), 
but these two conditions do not imply that the fraction field is global (Remark \ref{r:gldm}). It is not clear to us if these more relaxed hypotheses can be of much interest; however proofs are not burdened by such generalization and one can gain some conceptual clarity. \footnote{One could easily work in greater generality: see \cite{fptw} for an extension to any projective limit of compact rings or groups.} We will always state explicitly when we require $D$ to be a global Dedekind domain.
 
\subsubsection*{Style and readership} The subject of densities in $\N$ and its extensions to integers in number fields are of interest to a large number of mathematicians other than algebraic number theorists, so we thought it convenient to offer more explanations and proofs (especially in Section \ref{s:prelim}), in the hope of making our work accessible to a wider audience. As a consequence, algebraic number theorists might get bored by details on obvious facts, while readers with a different background might find our explanations too terse: we apologize to both groups.

\subsection*{Acknowledgments} This paper was started while the second author was visiting National Center for Theoretical Sciences, Mathematics Division, in Taipei and was completed while he was a guest of the Indian Institute of Science in Bangalore: he is grateful to both institutions for providing ideal conditions of work. Both authors also thank Ki-Seng Tan and National Taiwan University for the help in organizing a visit of the first author in Taipei to work on this paper.

We thank Paolo Leonetti for providing us with a preprint of \cite{lt} and Francesco Maria Saettone and Marc-Hubert Nicole for comments on a first draft of this paper. The second author thanks Lior Bary-Soroker for suggesting the idea used in the last counterexample in \S\ref{sss:cnt}. He also thanks Henri Darmon, Giuseppe Molteni, Alessandro Zaccagnini and Daniele Garrisi for useful discussions and he is grateful to Alejandro Vidal-L\'opez and Simon Lloyd for their patience in listening to him in too many occasions when thinking about this matter. Finally, he wishes to thank Marco Osmo for having awakened his interest in probabilistic number theory. Last but not least, thanks to the referee, whose suggestions resulted in a number of improvements (in particular, a better proof of Theorem \ref{t:chbt}).

\section{Notations and some prerequisites} \label{s:prelim}

\subsection{Basic notations and conventions} \label{ss:ntt} Throughout the paper, $D$ is a Dedekind domain and $F$ is its field of fractions. (The definition of Dedekind domain that we use is as in \cite[Chapter 9]{am} - in particular, the Krull dimension must be 1, so $D$ cannot be a field.) Ideals in $D$ will be usually denoted by German letters; in particular, $\gotp$ will always mean a non-zero prime ideal of $D$. When $F$ is a global field, we will write $F_\nu$ for the completion at a place $\nu$ and prime ideals of $D$ will often be tacitly identified with the corresponding place.

In \S\ref{sss:invlm} we will endow $D$ with a topology; the expression ``closed subset of $D$'' will always be used in reference to this topology. Also, finite sets will always have the discrete topology; products, quotients and limits of topological spaces will always be given the corresponding topology.

On various occasions we are going to take infinite sums or products over some countable set of indexes $J$. Usually we do not choose an ordering on $J$: the sum or product is meant to be taken as a limit with respect to the cofinite filter on $J$.

The cardinality of a set $X$ will be denoted by $|X|$ and we shall write $X-Y$ for the set-theoretic difference. Usually we will not distinguish between canonical isomorphisms and equalities.

See also \S\ref{sss:ntt} for more notations.

\subsection{The ring $\da$} Let $D$ be a Dedekind domain. We will denote the set of all non-zero ideals of $D$ by $\cali(D)$ and the subset of non-zero prime ideals by $\calp(D)$. For 
$\gotp\in\calp(D)$, let $v_\gotp\colon\cali(D)\rightarrow\N$ be the corresponding valuation, so that for $\gota\in\cali(D)$ we have the prime factorization $\prod_\gotp\gotp^{v_\gotp(\gota)}$.\\

\subsubsection{Inverse limits} \label{sss:invlm} We define
\begin{equation}\label{e:dfnda} \da:=\varprojlim D/\gota \end{equation}
where the limit is taken over all the non-zero ideals of $D$. We briefly recall what this means (for a more complete discussion, see e.g.~\cite[Chapter 10]{am}). For every $\gota\in\cali(D)$ let $\pi_\gota\colon D\rightarrow D/\gota$ be the reduction modulo $\gota$. There is a natural ring homomorphism
$$D\hooklongrightarrow\prod_{\gota\in\cali(D)}D/\gota$$
given by the product of all the maps $\pi_\gota$. (Injectivity  follows from the equality $\cap_{\cali(D)}\gota=\{0\}$.) Every quotient $D/\gota$ is given the discrete topology, making their product a Hausdorff topological space. \footnote{The topology is not strictly necessary for the construction of $\da$, but it is convenient for us to have it from the start.} Then $\da$ is defined as the closure of the image of $D$ in $\prod D/\gota$. The ring operations are defined componentwise on the product and they are continuous with respect to the product topology: as a consequence, $\da$ is a topological ring. Each $\pi_\gota$ extends by continuity to a ring homomorphism $\hpi_\gota\colon\da\twoheadrightarrow D/\gota$. A base of open sets of $\da$ is provided by 
\begin{equation}\label{e:base} \{\hpi_\gota^{-1}(x)\mid\gota\in\cali(D),x\in D/\gota\}\,.\end{equation}
By construction there is a canonical injection of $D$ into $\da$ and in the following we will always think of $D$ as a (dense) subring of $\da$.\\

For every prime ideal $\gotp$ in $D$, we also have the $\gotp$-adic completion
$$\dap:=\varprojlim D/\gotp^n$$
obtained by the same construction as in \eqref{e:dfnda} with $\cali(D)$ replaced by the set of powers of $\gotp$. It is well-known that for any $\gotp\neq0$ the ring $\dap$ is a complete discrete valuation domain.

\begin{thm} \label{t:crt} There is a canonical isomorphism of topological rings
\begin{equation}\label{e:crt} \da\simeq\prod_{\gotp\in\calp(D)}\dap \,,\end{equation}
where the product runs over all non-zero prime ideals of $D$. \end{thm}

\begin{skprf} If $\gotm$ divides $\gotn$ as ideals of $D$, we get a diagram
\begin{equation} \label{e:crtdg} \begin{CD} D/\gotn @>\sim>> \prod_\gotp D/\gotp^{v_\gotp(\gotn)}\\
@VVV @VVV\\
D/\gotm @>\sim>> \prod_\gotp D/\gotp^{v_\gotp(\gotm)}  \end{CD} \end{equation}
where the horizontal maps are the isomorphisms from the Chinese Remainder Theorem and the vertical maps are induced by the inclusion $\gotn\subset\gotm$. Diagram \eqref{e:crtdg} commutes and one concludes by taking the inverse limit on both sides. (This works by general abstract nonsense.) \end{skprf} 

In particular, \eqref{e:crt} implies that for every $\gotp$ there is a canonical projection $\hpi_{\gotp^\infty}\colon\da\rightarrow \dap$. We denote its kernel by $\gotp^\infty$. One can easily observe $\gotp^\infty=\cap_n\gotp^n\da\,,$ which justifies the notation. Actually, for every $\gotp$, the ring $\dap$ is endowed with a discrete valuation, which, composing with 
$\hpi_{\gotp^\infty}$, yields a valuation
$$v_\gotp\colon\da\longrightarrow\N\cup\{\infty\}\,.$$
The ideal $\gotp^\infty$ consists exactly of those $x\in\da$ such that $v_\gotp(x)=\infty$. Note also that the canonical injection of $D$ into $\dap$ factors via the map $\hpi_{\gotp^\infty}$.

\subsubsection{Closed ideals of $\da$} For any ideal $I\subset\da$, the quotient $\da/I$ inherits the topology from $\da$. By well-known properties of topological groups, $\da/I$ is Hausdorff if and only if $I$ is closed: therefore in the following we shall only consider closed ideals.

\begin{lem} \label{l:clpr} An ideal of $\da$ is closed if and only if it is principal. \end{lem}

\begin{prf} We first prove that principal ideals are closed, by showing that they have open complement. Let $x,y\in\da$ and assume $y\notin x\da$. Then there is 
some prime ideal $\gotp$ such that $v_\gotp(y)<v_\gotp(x)$. This implies that $x\da$ is contained in the kernel of
$$\hpi_{\gotp^{v_\gotp(y)+1}}\colon\da\longrightarrow D/\gotp^{v_\gotp(y)+1}$$
and $y$ is not. By definition of the topology, the image is discrete and the map is continuous: hence  the complement of $\ker(\hpi_{\gotp^{v_\gotp(y)+1}})$ is an open neighbourhood of $y$.

As for the converse, we start with the observation that by \eqref{e:crt} one can express $x\in\da$ as $x=(x_\gotp)_\gotp$, with $x_\gotp\in \dap$. For any subset $S$ of $\calp(D)$, let 
$e_S=(e_{S,\gotp})_\gotp$ be defined by 
$$e_{S,\gotp}:=\begin{cases}1\quad\text{ if }\gotp\in S\,;\\ 0\quad\text{ if }\gotp\notin S\,.    \end{cases}$$
Then $e_S\da$ is a subring of $\da$, isomorphic to $\prod_{\gotp\in S}\dap$\,.

Also, for every non-zero prime $\gotp$ choose $\tilde u_\gotp\in D$ satisfying $v_\gotp(\tilde u_\gotp)=1$ and put $u_\gotp:=e_{\{\gotp\}}\tilde u_\gotp$\,. The subring $e_{\{\gotp\}}\da$ is a discrete valuation domain having $u_\gotp$ as a uniformizer.
 
Let $I$ be any ideal of $\da$. By the above, $e_{\{\gotp\}}I$ is a principal ideal and we have $e_{\{\gotp\}}I=u_\gotp^{v_\gotp(I)}\da$ for some $v_\gotp(I)\in\N\cup\{\infty\}$. If $S$ is any finite set of non-zero primes, then the equality
$$e_S=\sum_{\gotp\in S}e_{\{\gotp\}}$$
implies
$$e_SI=\sum_{\gotp\in S}e_{\{\gotp\}}I=a_{S,I}\da\,,$$
with $a_{S,I}=\sum u_\gotp^{v_\gotp(I)}$. Moreover, $a_{S,I}\in I$, since $e_S\cdot I\subseteq I$.

Let $a_I\in\da$ be the point corresponding to $(\tilde u_\gotp^{v_\gotp(I)})_\gotp$ in the isomorphism \eqref{e:crt}. The inequality
$$v_\gotp(x)\ge v_\gotp(I)=v_\gotp(a_I)$$
holds for every $x\in I$ and every $\gotp$, proving the inclusion $I\subseteq a_I\da$.

Moreover $a_I$ is an accumulation point of the set $\{a_{S,I}\}$ (where $S$ varies among all finite subsets of non-zero primes). Indeed, let $U$ be any open neighbourhood of $a_I$. Without loss of generality, we can assume $U=\prod_\gotp U_\gotp$, where each $U_\gotp$ is open in $\dap$ and $U_\gotp=\dap$ for every $\gotp$ outside of a finite set $T$; but then $a_{S,I}\in U$ if $T\subseteq S$. If $I$ is closed this yields $a_I\in I$ and hence $a_I\da\subseteq I$. In the general case, one gets the equality $\widehat I=a_I\da$. \end{prf}

\begin{rmk} It might be worth observing that the key idea behind the proof of Lemma \ref{l:clpr} is the equality $1=\lim_Se_S$ (where the limit is taken with respect to the cofinite filter). \end{rmk}

\begin{eg} \label{eg:ncl} We provide an example of non-closed ideal of $\da$ (suggested by \cite[Theorem 1]{pst2}). Recall the isomorphism of topological rings \eqref{e:crt} and consider the set
$$I=\{(x_\gotp)_\gotp\in\da\mid x_\gotp=0\text{ for all but finitely many }\gotp\}\,.$$
It is immediate to check that $I$ is an ideal. By definition of the product topology, any open subset of $\da$ must contain a product $U=\prod_\gotp U_\gotp$, where every $U_\gotp$ is 
open in $\dap$ and $U_\gotp=\dap$ for almost all $\gotp$. Hence $U\cap I\neq\emptyset$, proving that $I$ is dense. Thus this ideal can be closed only if $I=\da$ and clearly this 
is not the case if $\calp(D)$ is infinite. On the other hand, if $D$ has only finitely many prime ideals, then one can take $S=\calp(D)$ in the proof of Lemma \ref{l:clpr} to show that 
every ideal of $\da$ is principal (and hence closed). \end{eg}

\subsubsection{Supernatural ideals} \label{sss:sprn} 

\begin{dfn} Let $\cals(D)$ be the set of all closed ideals of $\da$. We will call its elements the {\em supernatural ideals} of $D$. \end{dfn}

\begin{rmk} \label{r:sprnt} Elements of $\cals(\Z)$ are usually known as {\em supernatural numbers}: they can be described as formal products $\sigma=\prod_pp^{e_p}$, where the product is taken over all prime numbers and the exponents are allowed to take any value in $\N\cup\{\infty\}$. (Note that one can have $e_p\neq0$ for infinitely many $p$.) 

For a general $D$, it follows from the proof of Lemma \ref{l:clpr} that we can similarly write elements in $\cals(D)$ as $\sigma=\prod_\gotp\gotp^{e_\gotp}$, letting $\gotp$ vary in 
$\calp(D)$. If $\sigma=\gotp^\infty$, the map $\hpi_\sigma\colon\da\rightarrow\da/\sigma$ is exactly the projection $\hpi_{\gotp^\infty}\colon\da\rightarrow \dap$ that was described before.  \end{rmk}

The set of ideals of $D$ can be embedded in $\cals(D)$ by $\gota\mapsto\widehat\gota=\gota\da$. To lighten notation, in the following we will think of $\cali(D)$ as a subset of $\cals(D)$. Hence we shall often identify $\gota\in\cali(D)$ with $\widehat\gota\in\cals(D)$ as points in the space of supernatural ideals. (Of course we will distinguish between $\gota$ and $\widehat\gota$ as subsets of $\da$).

By Lemma \ref{l:clpr}, the map
\begin{equation} \label{e:ro} \rho\colon\da\longrightarrow\cals(D)\,,  \end{equation}
$x\mapsto x\da$, is surjective. Thus the set $\cals(D)$ inherits a topology as a quotient of $\da$. The notation
\begin{equation} \label{e:dflim0} \lim_{\sigma\rightarrow0} \end{equation}
will be used to mean that the limit is taken as $\sigma$ converges to $0$ in $\cals(D)$. (By a slight abuse of notation, we write $0$ to denote both the zero element in the ring $D$ and its image via $\rho$.) We remark that this topology makes $\cali(D)$ a dense subset of $\cals(D)$: actually, density can already be achieved just taking principal ideals, since $D-\{0\}$ is dense in $\da$. 

Moreover, $\cals(D)$ inherits a monoid structure from the product in $\da$, making $\cali(D)$ (with the usual product of ideals) a submonoid. There is also an order relation on $\cals(D)$, defined by
$$\sigma\le\tau\;\Longleftrightarrow\;\sigma|\tau\;\Longleftrightarrow\;\tau\da\subseteq\sigma$$ 
(following the usual convention for the order induced by divisibility of ideals in $D$).

\begin{eg} In the case $D=\Z$, the set $\N$ of natural numbers is a subset (in the obvious way) of the set of supernatural numbers $\cals(\Z)$. A basis of neighbourhoods of $0$ in the topological ring $\za$ is given by the set of ideals $\{a\za\}_{a\in\N}$\,: that is, $n$ gets closer and closer to $0$ as it has more and more divisors. So, if $(x_n)$ is a sequence taking values in a topological space $X$, the equality
$$\lim_{n\rightarrow0}x_n=y$$
means that for every neighbourhood $U$ of $y$ there is an $a\in\N$ such that $x_n\in U$ if $n\in a\za$ (equivalently, since both $a$ and $n$ are in $\N$, if $a$ divides $n$). \end{eg}

The valuation maps $\{v_\gotp\}$ descend in an obvious way to $\cals(D)$, by $v_\gotp(\prod\gotq^{e_\gotq})=e_\gotp$. We also define the support of $\sigma\in\cals(D)$ as
$$\supp(\sigma):=\{\gotp\mid v_\gotp(\sigma)>0\}$$
and the functions $\omega,\Omega\colon\cals(D)\rightarrow\N\cup\{\infty\}$,
\begin{equation} \label{e:ommin} \omega(\sigma):=|\supp(\sigma)| \end{equation}
and
\begin{equation} \label{e:Ommin} \Omega(\sigma):=\sum_\gotp v_\gotp(\sigma)\,, \end{equation}
which generalize the homonymous functions from elementary number theory. It might be worth to note the equality
$$\cali(D)=\{\sigma\in\cals(D)\mid\Omega(\sigma)<\infty\}\,.$$

\begin{rmk} \label{r:alxcmp} Give to $\N$ the discrete topology and let $\bar\N:=\N\cup\{\infty\}$ be its Alexandroff compactification. Remark \ref{r:sprnt} can be reformulated as saying that the valuation maps $\{v_\gotp\}$ induce a bijection $\sigma\mapsto\big(v_\gotp(\sigma)\big)_\gotp$ from $\cals(D)$ to $\bar\N^{\calp(D)}=\prod_{\gotp\in\calp(D)}\bar\N$. One can check that this map is actually a homeomorphism (where $\bar\N^{\calp(D)}$ has the product topology).

In the case $D=\Z$, this is the approach to supernatural numbers taken in the (very readable) paper \cite{pll}. Note also that supernatural numbers are a compactification of $\N$ (as 
emphasized in \cite{pll}) and that the latter is in natural bijection with the set of ideals of $\Z$. Under Assumption \ref{a:cntb} below, $\cals(D)$ is compact and we will think of it as a compactification of $\cali(D)$.  \end{rmk} 

\subsubsection{Compactness of $\da$} From now on, we shall assume the following hypothesis.

\begin{ass} \label{a:cntb}\begin{itemize}\item[]
\item[(A1)] $D$ is countable;
\item[(A2)] all non-zero ideals of $D$ have finite index.
\end{itemize}\end{ass}
\noindent From condition (A2) it follows that $\da$ is a compact topological ring. As for (A1), it implies that the set of all ideals of $D$ is also countable, because every ideal in a Dedekind domain can be generated by two elements. In particular also the set of prime ideals of $D$ is countable. (In the following, we shall usually assume that $D$ has infinitely many prime ideals.)

The two conditions together imply that $\da$ is a second-countable topological space (that is, it has a countable base). Indeed, cosets of ideals of $D$ induce a base for the topology on $\da$, as expressed in \eqref{e:base}.

\begin{rmk} \label{r:gldm} It is well-known that Assumption \ref{a:cntb} holds when $D$ is a global Dedekind domain or a localization of such, but it is interesting to note that the converse is false: there are countable, residually finite Dedekind domains whose fraction field is not global. A nice (although non constructive) example is provided by \cite[Theorem on page 114]{gld}. This result implies that there exists a ring $D$ containing $\Z[x]$ (where $x$ is a transcendental indeterminate) and having $\Q(x)$ as field of quotients, which is a Dedekind domain and satisfies (A2) (as for (A1), it is obvious since $\Q(x)$ is countable). \end{rmk}

For future reference, we recall some results about open ideals of $\da$. 

\begin{lem} \label{l:msr0} An ideal of $\da$ has finite index if and only if it is open. \end{lem}

\begin{prf} Since $\da$ is a topological ring, the map $x\mapsto x+a$ is a homeomorphism for every $a\in\da$. Thus if an ideal $I$ is open then so are all its cosets. Hence $\da/I$ is a discrete compact topological space, which must be finite. 

Vice versa, if $I$ is an ideal of $\da$ of finite index, then the natural map $D\rightarrow\da/I$ must have a non-trivial kernel $\gota$, because $D$ is infinite. Since $I$ is an ideal of $\da$ and $\gota\subset I$, also the open ideal $\widehat\gota=\gota\da$ must be in $I$. Being a union of open cosets of $\widehat\gota$, $I$ must be open. \end{prf}

\begin{lem} \label{l:opn} Every open ideal of $\da$ is the completion of an ideal of $D$ with the same index. \end{lem}

\begin{prf} Let $I$ be an open ideal of $\da$. Then $I$ is closed, since all its cosets are open. The set $D\cap I$ is an ideal of $D$ and it is dense in $I$ (because $D$ is dense in $\da$ and $I$ is open). Finally, we get $D/(I\cap D)\simeq\da/I$ observing that the image  of $D$ in $\da/I$ is dense. \end{prf}

\begin{lem} \label{l:cmpopn} A subset of $\da$ is both open and closed if and only if it is a finite union of cosets of some open ideal $\widehat\gota$. \end{lem}

\begin{prf} Assume $C\subset\da$ is both open and closed. Being open, it is a union of subsets of the base \eqref{e:base}. Being closed, it is also compact and thus a finite union is enough. If $C=\cup_{i=1}^n\hpi_{\gota_i}^{-1}(x_i)$, then one can take $\gota=\cap\gota_i$ to get $C=\hpi_\gota^{-1}(A)$ for some $A\subseteq D/\gota$. \end{prf}

\begin{rmk} Under Assumption \ref{a:cntb}, it is possible to define a metric on $D$ so that $\da$ is the completion with respect to it. We mention this because some readers might feel more 
comfortable thinking of $\da$ as a metric space. However the topological structure is sufficient for the goals of this paper and it provides more straightforward arguments than those we would get from imposing a distance. \end{rmk}

\subsubsection{Some notations} \label{sss:ntt} The following notations shall be used throughout this paper: \begin{itemize}
\item for $\sigma\in\cals(D)$ and $n$ any positive integer, 
$$\hpi_\sigma\colon\da^n\longrightarrow(\da/\sigma)^n$$
is the natural projection;
\item for $\sigma\in\cals(D)$ and $X\subseteq\da^n$,
\begin{equation} \label{e:dfXsgm} X_\sigma:=\hpi_\sigma^{-1}\big(\overline{\hpi_\sigma(X)}\big)\,, \end{equation}
where $\overline{\hpi_\sigma(X)}$ is the closure of $\hpi_\sigma(X)$ in $(\da/\sigma)^n$ (with respect to the quotient topology);
\item $\xa$ is the closure of $X\subseteq\da^n$
\item for $I$ any ideal of $\da$, its index is denoted
\begin{equation} \label{e:indx} \|I\|:=|\da/I|\,\in\,\N\cup\{\infty\}\; \end{equation}
and given $a\in\da$ we use the shortening $\|a\|$ for $\|a\da\|$;
\item $\da^*$ is the group of units of $\da$ (not to be confused with $\wds$, the closure of $D^*$ in $\da$). \end{itemize}

In the case of an ideal $\gota$ of $D$, it is easy to check that one has $\widehat\gota=\gota\da$; moreover the equality
$$D/\gota=\da/\,\widehat\gota$$
holds for every $\gota\in\cali(D)$. Because of the embedding $\cali(D)\hookrightarrow\cals(D)$, we write $\hpi_\gota$ for $\hpi_{\widehat\gota}$ and we have $X_\gota=\hpi_\gota^{-1}(\hpi_\gota(X))$.

We shall occasionally use $\pi_\gota$ to abbreviate $\hpi_\gota|_{D^n}$\,.  

\begin{eg} Readers might attain a better grasp on these objects by observing that if $n=1$ then $X_\gota\cap D$ is simply a disjoint union of cosets of $\gota$. In the $D=\Z$, $n=1$ case, this is a union of arithmetic progressions. When $X=\{x\}$ is a singleton, we have
$$X_\gota=\prod_\gotp(x+\gota \dap^n)=\prod_{\gotp|\gota}(x+\gotp^{v_\gotp(\gota)}\dap^n)\times\prod_{\gotp\nmid\gota}\dap^n$$
for any $n$ and any $\gota$. Another example which is useful to have in mind is the following: fix a prime $\gotp$ and let $X(\gotp)$ denote the closure of $X$ in $\dap^n$. Then
\begin{equation} \label{e:Xpinf} X_{\gotp^\infty}=\hpi_{\gotp^\infty}^{-1}\big(X(\gotp)\big)=X(\gotp)\times\prod_{\gotq\neq\gotp}\hat D_\gotq^n \end{equation}
because $\ker(\hpi_{\gotp^\infty})=\{0\}\times\prod\hat D_\gotq^n$.  \end{eg}

\begin{rmk} By Lemmata \ref{l:msr0} and \ref{l:opn}, we obtain $\|I\|=\infty$ unless $I=\widehat\gota$ for some $\gota\in\cali(D)$, in which case 
$$\|\widehat\gota\|=|\da/\,\widehat\gota|=|D/\gota|\,.$$
\end{rmk}

Note also that the divisibility relation $\sigma|\tau$ implies $X_\tau\subseteq X_\sigma$ for every $X\subseteq\da^n$. 

\subsubsection{Closures in $\da^n$} Let $n$ be a positive integer. From now on we shall consider subsets of $D^n$ and $\da^n$. Note that the family \eqref{e:base} provides a base for the product topology on $\da^n$ for every $n$ (with the obvious change, with respect to \eqref{e:base}, of the domain of $\hpi_\gota$ from $\da$ to $\da^n$).

\begin{lem} \label{l:clsr} Let $\calt$ be a subset of $\cals(D)$. Then for every subset $X$ of $\da^n$, we have 
\begin{equation} \label{e:clsr} \xa=\bigcap_{\sigma\in\calt}X_\sigma\, \end{equation}
if $0$ is an accumulation point of $\calt$. \end{lem}

\begin{prf} By definition each $X_\sigma$ is a closed set containing $X$. Hence $\xa$ is contained in the intersection on the right-hand side of \eqref{e:clsr}.

Vice versa, let $z\in\da^n$ be in the complement of $\xa$. By definition of the topology on $\da^n$, there is an ideal $\gota$ of $D$ such that $(z+\gota\da^n)\cap\xa=\emptyset$ - that is, $\hpi_\gota(z)\notin\hpi_\gota(X)$. The assumption on $\calt$ implies that there is some $\sigma\in\calt$ such that $\sigma\subseteq\widehat\gota$. Hence $\hpi_\sigma(z)\notin\hpi_\sigma(X)$, so $z\notin X_\sigma$. This shows that $z$ is not in the right-hand side of \eqref{e:clsr}.  \end{prf}

\begin{rmks} \label{r:idl} \begin{itemize} \item[]
\item[{\bf 1.}] Recall that ideals of $D$ form a dense subset of $\cals(D)$. Therefore in Lemma \ref{l:clsr} one can take $\calt=\cali(D)$ to get $\xa=\cap X_\gota$.
\item[{\bf 2.}] From Assumption \ref{a:cntb} we get that $\cals(D)$ is second-countable. In particular, $0$ has a countable neighbourhood basis: it follows that for any $\calt$ having $0$ as an accumulation point there is a countable set $\calt'\subseteq\calt$ having the same property. \end{itemize} \end{rmks}

The following immediate consequence of Lemma \ref{l:clsr} provides an easy way of checking if a set is dense.

\begin{cor} A subset $X$ of $\da^n$ is dense if there is some $\calt\subseteq\cals(D)$ having $0$ as an accumulation point and such that $X_\sigma=\da^n$ for every $\sigma\in\calt$. Vice versa, if $X$ is dense then $X_\sigma=\da^n$ for every $\sigma\in\cals(D)$. \end{cor}

For example, a subset of $\N$ is dense in $\za$ if and only if it surjects onto $\Z/n\Z$ for every $n$.

\section{The closure of prime irreducibles} \label{s:clsirr} In this section we assume that $D$ is a global Dedekind domain.\\

Let $\Irr(D)$ denote the set of the prime elements of $D$, i.e., those irreducible elements which generate a prime ideal: that is,
$$\Irr(D):=\big\{x\in D\mid\Omega(x\da)=1\big\}$$
where $\Omega$, as defined in \eqref{e:Ommin}, is the function counting prime divisors of a supernatural ideal (with multiplicity). Note that one can write
\begin{equation} \label{e:irrdp} \Irr(D)=\bigsqcup_{\gotp\in\calp(D)}\Irr(\gotp)\,, \end{equation}
where $\Irr(\gotp):=\{p\in D\mid pD=\gotp\}$ is empty if $\gotp$ is not principal. We will compute the closure of $\Irr(D)$ in $\da$.

\begin{rmk} For a general Dedekind domain, the set $\Irr(D)$ might be empty, even when Assumption \ref{a:cntb} holds: actually, as shown in \cite[Example 1-5]{clab}, one can start with any $D$ and localize it so to kill all principal prime ideals and preserve the class group. \end{rmk}

Let $\wds $ denote the closure of $D^*$ in $\da$. We shall need
$$\wds\Irr(D)=\{ux\mid u\in\wds , x\in\Irr(D)\}.$$

\begin{thm} \label{t:chbt} If $D$ is a global Dedekind domain, then
\begin{equation} \label{e:clirr} \widehat{\Irr(D)}=\da^*\sqcup\wds\Irr(D)\,. \end{equation} \end{thm}

\begin{prf} The right-hand side of \eqref{e:clirr} is indeed a disjoint union, because one has
$$\Omega(x\da)=\begin{cases} 0 \;\;\text{ if }x\in\da^*\\ 1 \;\;\text{ if }x\in\wds\Irr(D)\,. \end{cases}$$

The easy part of proving \eqref{e:clirr} is to show the inclusions
$$\wds\Irr(D)\subseteq\widehat{\Irr(D)}\subseteq\da^*\cup\wds\Irr(D)\,,$$ 
which amount to the statements \begin{itemize}
\item[\bf Irr1\,:] every element in $\wds\Irr(D)-\Irr(D)$ is a limit point of $\Irr(D)$;
\item[\bf Irr2\,:] $\Irr(D)$ has no accumulation point outside $\da^*\cup\wds\Irr(D)$. \end{itemize}

The claim {\bf Irr1} is obvious: if $\Irr(\gotp)$ is not empty then it is a $D^*$-orbit and \eqref{e:irrdp} shows that $\wds\Irr(D)$ is the union of the closures of these orbits.

As for {\bf Irr2}, we show that $x\in\da-\da^*$ can be an accumulation point of $\Irr(D)$ only if it is in $\wds\Irr(D)$. If $\Omega(x\da)>1$, then there are two prime ideals $\gotp$ and 
$\gotq$ (possibly equal) such that $x\in\gotp\gotq\da$; but then $\gotp\gotq\da$ is a neighbourhood of $x$ and clearly it cannot contain any element of $\Irr(D)$. So we can assume $\Omega(x\da)=1$. In this case $x\da=\widehat\gotp$ for some $\gotp\in\calp(D)$, which implies that $\widehat\gotp$ is a neighbourhood of $x$. If $\gotp$ is not principal then $\widehat\gotp$ has empty intersection with $\Irr(D)$. Therefore $x$ can be an accumulation point of $\Irr(D)$ only if $\widehat\gotp=p\da$ for some $p\in\Irr(D)$. But then $\widehat\gotp\cap\Irr(D)=D^*p$, so if $x$ is an accumulation point it must be in $\wds p$, which is the closure of $D^*p$.

The hard part is to show that every unit of $\da$ is a limit of prime irreducibles. In Lemma \ref{l:chbt} below, we will prove that for any $\gota\in\cali(D)$ we have 
$(D/\gota)^*\subseteq\pi_\gota\big(\Irr(D)\big)$. This is enough, because $\da^*=\varprojlim(D/\gota)^*$ yields $\hpi_\gota(\da^*)=(D/\gota)^*$ and, by Lemma \ref{l:clsr}, 
$$\widehat{\Irr(D)}=\bigcap_{\gota\in\cali(D)}\Irr(D)_\gota=\bigcap_{\gota\in\cali(D)}\hpi_\gota^{-1}\big(\pi_\gota(\Irr(D))\big)\,.$$
\end{prf}

\begin{rmk} \label{r:drc} In the case $D=\Z$, letting $\calp\subset\N$ denote the set of prime numbers, it is straightforward (using just the definition of the topology on $\za$) to see that the claim $\za^*\subseteq\cpr$ is equivalent to Dirichlet's theorem on primes in arithmetic progressions. \end{rmk}

Our next goal is to prove the equality \begin{equation} \label{e:pairr} \bigcup_{\gotp\notin\supp(\gota)}\pi_\gota\big(\Irr(\gotp)\big)=(D/\gota)^*\,. \end{equation}
As suggested by Remark \ref{r:drc}, the main tool will be a generalization of Dirichlet's result. For lack of a convenient reference, we shall sketch a proof of it in Theorem \ref{t:rfr}, after having first recalled some facts about ideles.

\subsubsection{A reminder on ideles} By a canonical identification, we can think of $\calp(D)$ as a subset of $\calv(F)$, the set of all places of $F$. Also, let $S_D$ be the complement of 
$\calp(D)$ in $\calv(F)$. The idele group of $F$ is defined as
$$\idl_F:=\varinjlim_S\,\prod_{\gotp\notin S}\dap^*\times\prod_{v\in S}F_v^*=\varinjlim_S\idl_S\,,$$
where the limit is taken over all finite subsets $S\subset\calv(F)$ containing $S_D$\,. Each $\idl_S$ is given the product topology: hence if $S\subseteq T$ the map $\idl_S\hookrightarrow\idl_T$ in the direct system is an open embedding. There is an obvious  injection of $\idl_F$ into $\prod_{v\in\calv(F)}F_v^*$ (not preserving the topology) and in the following we will use it to describe a generic idele as $x=(x_v)_{v\in\calv(F)}$.\\

We will need certain subgroups of $\idl_F$. First of all, $F^*$ will be thought of as a discrete subgroup of $\idl_F$ via the diagonal embedding: the idele class group is the quotient $\idl_F/F^*$.

For any place $v$ of $F$, let $\iota_v\colon F_v^*\hookrightarrow\idl_F$ be the map which sends $x\in F_v^*$ in the idele having component $x$ at $v$ and $1$ otherwise: that is,
$$ (\iota_v(x))_w:=\begin{cases} 1 \;\text{ if }v\neq w\,; \\ x\;\text{ if }v=w\,. \end{cases}$$
By \eqref{e:crt}, the product of the maps $\iota_\gotp|_{\dap^*}$ for all $\gotp\in\calp(D)$ defines an embedding $\iota_D\colon\da^*\hookrightarrow\idl_F$. The maps $v_\gotp$ extend to $\idl_F$ in the obvious way. Since $F^*\cap\iota_D(\da^*)=\{1\}$, the homomorphism $\iota_D$ induces an injection 
$$\bar\iota_D\colon\da^*\hooklongrightarrow\idl_F/F^*\,.$$
Another subgroup of $\idl_F$ that we are going to use is
$$F_\infty^*:=\prod_{v\in S_D}\iota_v(F_v^*)$$
together with the injection $\iota_\infty\colon F^*\hookrightarrow F_\infty^*$ defined as the product of $\iota_v|_{F^*}$ for all $v\in S_D$\,. Finally, we consider the subgroups 
$\iota_D(U_\gota)$, where $\gota\in\cali(D)$ and
$$U_\gota:=\{x\in\da^*\mid x\equiv1\!\!\mod\widehat\gota\}=(1+\gota\da)\cap\da^*\,.$$
We want to assess the image of $\da^*$ in the quotient group $\idl_F/F^*F_\infty^*\iota_D(U_\gota)\,.$

\begin{lem} For any $\gota\in\cali(D)$, we have 
\begin{equation} \label{e:intrsidl} \iota_D(\da^*)\cap F^*F_\infty^*\iota_D(U_\gota)=\iota_D(D^*U_\gota)\,. \end{equation} \end{lem}

\begin{prf} Assume $\iota_D(x)=ay\,\iota_D(u)$ with $x\in\da^*$, $a\in F^*$, $y\in F_\infty^*$ and $u\in U_\gota$. Comparing components at places in $S_D$ yields $y=\iota_\infty(a^{-1})$ 
and computing valuations one gets $v_\gotp(a)=0$ for every $\gotp\in\calp(D)$. Therefore $a$ is in $D^*$ and $x=au$ because $\iota_D$ is injective.

Vice versa, given $a\in D^*$ and $u\in U_\gota$, it is clear that $\iota_D(au)=a\iota_\infty(a^{-1})\iota_D(u)\in F^*F_\infty^*\iota_D(U_\gota)$. \end{prf}

Equality \eqref{e:intrsidl} shows that $\iota_D$ induces an injection $\da^*/D^*U_\gota\hooklongrightarrow\idl_F/F^*F_\infty^*\iota_D(U_\gota)\,.$ By \eqref{e:crt} we can write
$$U_\gota=\prod_{\gotp\in\supp(\gota)}(1+\gotp^{v_\gotp(\gota)}\dap)\times\prod_{\gotp\notin\supp(\gota)}\dap^*$$
which shows that there is an exact sequence
$$\begin{CD} 0 @>>> U_\gota @>>> \da^* @>\hpi_\gota>> (D/\gota)^* @>>> 0\,. \end{CD}$$
Hence there  is a surjection
\begin{equation} \label{e:psia} \psi_\gota\colon(D/\gota)^*\longrightarrow\!\!\!\!\!\rightarrow\da^*/D^*U_\gota \end{equation}
with kernel $\pi_\gota(D^*)$.

Let $\phi_\gota\colon\idl_F/F^*\twoheadrightarrow\idl_F/F^*F_\infty^*\iota_D(U_\gota)$ be the quotient  map. Summarizing, we have a commutatitive diagram
\begin{equation} \label{e:cmmidl} \begin{CD} && \da^* @>\bar\iota_D>> \idl_F/F^*   \\
&& @VV{\psi_\gota\circ\hpi_\gota}V @VV{\phi_\gota}V  \\
1 @>>> \da^*/D^*U_\gota @>\text{via }\iota_D>> \idl_F/F^*F_\infty^*\iota_D(U_\gota) @>>> \idl_F/F^*F_\infty^*\iota_D(\da^*) @>>> 1 \vspace{3pt}  \end{CD} \end{equation} 
where the second line is exact and vertical maps are surjective. Let $\cls(D)$ be the class group of $D$. Considering the map which sends $x\in\idl_F/F_\infty^*$ to $\prod_\gotp\gotp^{v_\gotp(x)}$, one obtains an isomorphism
\begin{equation} \label{e:cls} \idl_F/F^*F_\infty^*\iota_D(\da^*)\simeq\cls(D)\,. \end{equation}

\subsubsection{Conclusion of the proof of Theorem \ref{t:chbt}} For every $\gotp\in\calp(D)$, choose a uniformizer $t_\gotp\in F_\gotp$ and define a map $\iota_\calp\colon\calp(D)\rightarrow\idl_F$ by $\gotp\mapsto\iota_\gotp(t_\gotp)$.

\begin{thm} \label{t:rfr} The composition $\beta=\phi_\gota\circ\iota_\calp$ is surjective. \end{thm}

\begin{skprf} To lighten notation, let $G$ denote the group $\idl_F/F^*F_\infty^*\iota_D(U_\gota)$ and $G^\vee$ its dual. Note that $G$ is finite, by \eqref{e:psia} and \eqref{e:cls}. To each $\chi\in G^\vee$, attach an $L$-function
$$L(\chi,s)=\prod_{\gotp\in\calp(D),\,\gotp\nmid\gota}\left(1-\frac{\chi(\beta(\gotp))}{\|\gotp\|^s}\right)^{-1}\,.$$
It is well-known that all the functions $L(\chi,s)$ are meromorphic, with a pole in $s=1$ if $\chi$ is trivial and neither poles nor zeroes in $1$ otherwise (see e.g. \cite[VII, \S5, Corollary 2 and page 293, note P.126]{wl}). From here one can reason as in \cite[VI, \S4, proof of Theorem 2]{ser}. Namely, let $\calp_g(D)$ be the set of those $\gotp$ such that $\beta(\gotp)=g$. One concludes noticing that
$$\sum_{\gotp\in\calp_g(D)}\frac{1}{\|\gotp\|^s}=\frac{1}{|G|}\sum_{\chi\in G^\vee}\chi(g^{-1})\sum_{\gotp\in\calp(D)}\frac{\chi(\beta(\gotp))}{\|\gotp\|^s}\sim\log\prod_{\chi\in G^\vee}L(\chi,s)$$
(where $\sim$ refers to the asymptotic behaviour as $s$ approaches $1$) implies the divergence of the leftmost sum in $s=1$. Therefore $\calp_g(D)$ is non-empty. \end{skprf}

\begin{lem} \label{l:chbt} Equality \eqref{e:pairr} is true for every $\gota\in\cali(D)$. \end{lem}

\begin{prf} Let $\gotp$ be a principal prime and $u_\gotp$ any element in $\Irr(\gotp)$. It is obvious that $\gotp\notin\supp(\gota)$ implies $\pi_\gota(u_\gotp)\in(D/\gota)^*$.

In order to prove  the other inclusion, take $w\in(D/\gota)^*$ and recall the surjection $\psi_\gota$ from \eqref{e:psia}. By diagram \eqref{e:cmmidl} and Theorem \ref{t:rfr}, there is some prime $\gotp$ outside $\supp(\gota)$ such that
$$\psi_\gota(w)=\phi_\gota\big(\iota_\calp(\gotp)\big)^{-1}\,.$$
(The reason for taking the inverse will soon be clear.) Comparing the lower line of \eqref{e:cmmidl} with \eqref{e:cls} shows that $\gotp$ is a principal ideal: let $u_\gotp\in F$ be a generator. We can think of $u_\gotp$ as an idele via the diagonal embedding of $F^*$ into $\idl_F$; moreover, since $F$ is a subfield of each $F_v$, we can consider the ideles $\iota_v(u_\gotp)$ for every $v$. In particular, let $t_\gotp:=\iota_\gotp(u_\gotp)$. Note that the idele
$$x_\gotp:=u_\gotp^{-1}\prod_{v|\infty}\iota_v(u_\gotp)=u_\gotp^{-1}\iota_\infty(u_\gotp)$$
belongs to $F^*F_\infty^*$, which implies
\begin{equation} \label{e:frobp} \phi_\gota(t_\gotp)=\phi_\gota(t_\gotp x_\gotp)\,. \end{equation}
Furthermore $t_\gotp x_\gotp=\iota_D(z_\gotp)$ for some $z_\gotp\in\da^*$, because the idele $t_\gotp x_\gotp$ has components
\begin{equation} \label{e:cmptpxp} (t_\gotp x_\gotp)_v=\begin{cases} \begin{array}{ll} 1 & \text{ if }v\in S_D\cup\{\gotp\} \\ 
u_\gotp^{-1} & \text{ if }v\notin S_D\cup\{\gotp\} \end{array}\end{cases} \end{equation}
and $v_\gotq(u_\gotp)=0$ for each $\gotq\in\calp(D)-\{\gotp\}$. The image in $(D/\gota)^*$ only depends on the residue class at primes in $\supp(\gota)$ and thus \eqref{e:cmptpxp} yields 
$\pi_\gota(u_\gotp)=\hpi_\gota(z_\gotp^{-1})$. Therefore we obtain
$$\psi_\gota\big(\pi_\gota(u_\gotp)\big)=\psi_\gota\big(\hpi_\gota(z_\gotp^{-1})\big)\stackrel{\text{by \eqref{e:cmmidl}}}{=}
\phi_\gota\big(t_\gotp^{-1}x_\gotp^{-1}\big)\stackrel{\text{by \eqref{e:frobp}}}{=}\psi_\gota(w)\,.$$
This implies $w\in\pi_\gota(D^*u_\gotp)$, because the kernel of $\psi_\gota$ is $\pi_\gota(D^*)$. The observation $D^*u_\gotp=\Irr(\gotp)$ concludes the proof. \end{prf}

\subsubsection{The closure of $D^*$} \label{sss:clD*} Theorem \ref{t:chbt} expresses the closure of $\Irr(D)$ as a disjoint union of two pieces, $\da^*$ and $\wds\Irr(D)$. If $D^*$ is finite, then $\wds\Irr(D)=\Irr(D)$ is countable and hence it is much ``smaller'' (in an obvious sense) than the uncountable set $\da^*$. We are going to show that this is always the case, for any global Dedekind domain $D$.

\begin{prop} \label{p:indunt} If $D$ is a global Dedekind domain, then $|\da^*/\wds |=\infty$.  \end{prop}

\begin{prf} Let $\F_\gotp:=\dap/\gotp$ denote the residue field of $\gotp\in\calp(D)$. For every integer $n>1$, put 
$$\calp_n(D):=\big\{\gotp\in\calp(D)\mid n\text{ divides }|\F_\gotp^*|\big\}$$
and $(\F_\gotp^*)^n:=\{x^n\mid x\in\F_\gotp^*\}$. If $\gotp\in\calp_n(D)$ then the quotient $\F_\gotp^*/(\F_\gotp^*)^n$ is a cyclic group of order $n$. The quotient homomorphism
$\dap^*\rightarrow\F_\gotp^*/(\F_\gotp^*)^n$ is onto for every $\gotp$ and thus, by \eqref{e:crt}, we obtain a surjection
$$\varphi_n\colon\da^*\longrightarrow\!\!\!\!\!\rightarrow\prod_{\gotp\in\calp_n(D)}\F_\gotp^*/(\F_\gotp^*)^n\,.$$

By Dirichlet's unit theorem (in its $S$-units version) $D^*$ is a finitely generated group of rank $|S_D|-1$ and with cyclic torsion: hence $\varphi_n(D^*)$ is at most a product of $|S_D|$ cyclic groups of order $n$ and therefore
$$|\varphi_n(D^*)|\le n^{|S_D|}\,.$$ 
The image $\varphi_n(D^*)$ is dense in $\varphi_n(\wds )$ and, being finite, it is closed: hence $\varphi_n(\wds )=\varphi_n(D^*)$. On the other hand, if $\calp_n(D)$ is infinite then 
$\varphi_n(\da^*)$ is uncountable: in this case,
$$|\da^*/\wds |\ge|\varphi_n(\da^*)/\varphi_n(\wds )|=\infty\,.$$

Thus we just need to show that $\calp_n(D)$ is infinite for at least one $n$. Let $\car(F)$ denote the characteristic of $F$. The set $\calp_2(D)$ is infinite if $\car(F)\neq2$. If $\car(F)=2$, let $\F$ be the constant field of $F$. If $|\F|>2$, then one can take $n=|\F^*|$ since $\F\subseteq\F_\gotp$ for every $\gotp$. The reasoning becomes more sophisticated if $\F=\F_2$ - for example, one can show that $\calp_3(D)$ is infinite by proving that there are infinitely many primes of $D$ which split in the constant field extension $K=F\F_4$. As well known, this can be seen by observing that the Dedekind zeta function $\zeta_K(s)$ has a pole in $s=1$, but the product $\prod(1-\|\gotp\|^{-s})^{-1}$ over places $\gotp$ of $K$ which are inert in $K/F$ converges (because so does the sum $\sum\|\gotp\|^{-s}$ over the same set of places). \end{prf}

Restricting the topology of $\da$ makes $\da^*$ a compact topological group (this is the same topology on $\da^*$ as induced by the embedding in $\idl_F$). In particular, $\da^*$ has its own Haar measure $\mu_{\da^*}$, normalized so to have total mass 1.

\begin{cor} \label{c:irrbpcc} The set $\wds$ has empty interior in the topology of $\da^*$ and $\mu_{\da^*}(\wds)=0$.  \end{cor}

\begin{prf} The Haar measure of a closed subgroup is the inverse of its index; moreover every open subset has positive measure. Thus the first statement follows from the second, which in 
turn is immediate from Proposition \ref{p:indunt}. \end{prf}

By \eqref{e:irrdp}, the set $\wds\Irr(D)$ is the disjoint union of countably many copies of $\wds $. Putting together Theorem \ref{t:chbt} and Corollary \ref{c:irrbpcc}, one can say that the ``bulk'' of $\Irr(D)$ is somehow measured by $\da^*$. We hope to show in future work how this can be used to study the relative density of subsets of the primes.

\begin{rmk} It is straightforward to check that one has $D\cap\wds\Irr(D)=\Irr(D)$ and $D\cap\da^*=D^*$. Hence Theorem \ref{t:chbt} implies that $\Irr(D)$ is not a closed subset of $D$ in the induced topology: its closure is exactly $\Irr(D)\sqcup D^*$ (and the closure of $\calp$ in $\Z$ is $\calp\sqcup\{\pm1\}$). Corollary \ref{c:irrbpcc} shows that the part $D^*$ can be considered negligible. \end{rmk}

\section{Measure and density} \label{s:m&d}

In this section $D$ is any Dedekind domain satisfying Assumption \ref{a:cntb}.

\subsection{The measure} Since $\da$ is a compact topological ring, there is a Haar measure $\mu$ on it, normalized so to have $\mu(\da)=1$. By definition $\mu$ is invariant under translation by any element of $\da$, hence all cosets of an ideal have the same measure and it follows that for any ideal $I$ of $\da$, one has 
\begin{equation} \label{e:dfnhaar} \mu(I)=\frac{1}{|\da/I|}=\frac{1}{\|I\|}\;. \end{equation}
Here and in the following, we use the convention $\frac{1}{\infty}=0$. By Lemma \ref{l:msr0}, $\mu(I)\neq0$ if and only if $I$ is open (that is, if $I$ is the completion of a non-trivial ideal of $D$).

Multiplication by $a\in\da$ changes the Haar measure by a scalar multiple and thus \eqref{e:dfnhaar} implies
\begin{equation} \label{e:msrprd} \mu(aX)=\frac{1}{\|a\|}\cdot\mu(X) \end{equation}
for every measurable $X\subseteq\da$.

By a slight abuse of notation, in the following we shall also denote by $\mu$ the measure on $\da^n$ induced by the measure on $\da$. Reasoning as above, we get
\begin{equation} \label{e:hrdn} \mu(\gota\da^n)=\frac{1}{\|\gota\|^n} \end{equation}
for every $\gota\in\cali(D)$.

\begin{lem} \label{l:msrlmt} For every $X\subseteq\da^n$ and $\calt\subseteq\cals(D)$ having $0$ as a limit point, 
\begin{equation} \label{e:limmu} \mu(\xa)=\lim_{\sigma\rightarrow0}\mu(X_\sigma)\;, \end{equation} 
where the limit is taken letting $\sigma$ vary in $\calt$. \end{lem}

\begin{prf} One has $\mu(\xa)\le\mu(X_\sigma)$ for every $\sigma$, because $\xa\subseteq X_\sigma$ holds by definition. The equality \eqref{e:limmu} is then obvious from 
\eqref{e:clsr} and Remark \ref{r:idl}.2. \end{prf}

In  Lemma \ref{l:msrlmt} one can take a subset of $\cali(D)$ as $\calt$. Since $\hpi_\gota(X)$ is a finite set for every $\gota\in\cali(D)$, equation \eqref{e:limmu} becomes
\begin{equation} \label{e:limidl} \mu(\xa)=\lim_{\gota\rightarrow0}\mu(X_\gota)=\lim_{\gota\rightarrow0}\frac{|\hpi_\gota(X)|}{\|\gota\|^n}\,. \end{equation}
(One can observe that \cite[Theorem 1]{pst} is just a special case of this.) As the next result shows, there are other interesting choices for $\calt$.

\begin{cor} \label{c:egprd} For any $\gotp$, let $C_\gotp\subseteq \dap^n$ be a closed set and put $C=\prod_\gotp C_\gotp$. Then
$$\mu(C)=\prod_{\gotp\in\calp(D)}\mu_\gotp(C_\gotp)\,,$$ 
where $\mu_\gotp$ denotes the Haar measure on $\dap^n$ (normalized so to have $\mu_\gotp(\dap^n)=1$). \end{cor}

\noindent Note that $C$ is closed (and therefore measurable) because each $C_\gotp$ is compact. Since all factors take value in $[0,1]$, the (possibly) infinite product on the right obviously has a limit.

\begin{prf} For $S$ a finite set of non-zero primes, put $\sigma_S:=\prod_{\gotp\in S}\gotp^\infty$. Then
$$C_{\sigma_S}=\prod_{\gotp\in S} C_\gotp\times\prod_{\gotp\notin S}\dap^n$$ 
and, by Fubini's theorem, $\mu(C_{\sigma_S})=\prod_{\gotp\in S}\mu_\gotp(C_\gotp).$ Let $\calt=\{\sigma_S\}_S$, where $S$ varies among all finite sets of non-zero primes. Then $\calt$ is a subset of $\cals(D)$ having $0$ as a limit point and one can apply Lemma \ref{l:msrlmt} to conclude. \end{prf}

\begin{eg} \label{eg:da*0} Corollary \ref{c:egprd} applies to the group of units of $\da$, since $\da^*=\prod \dap^*$ by \eqref{e:crt}. For computing $\mu(\da^*)$, it is convenient to introduce $\zeta_D$, the Dedekind zeta function of $D$, formally defined by the usual Euler product
\begin{equation} \label{e:Ddkzt} \zeta_D(s):=\prod_{\gotp\in\calp(D)}\left(1-\frac{1}{\|\gotp\|^s}\right)^{-1}. \end{equation}
(Of course there is no reason for the product in \eqref{e:Ddkzt} to converge for any $s\in\R$ without some stronger hypothesis on $D$, like it being a global Dedekind domain.)

One has $\dap^*=\dap-\gotp \dap$ and hence $\mu_\gotp(\dap^*)=\mu_\gotp(\dap)-\mu_\gotp(\gotp \dap)$. Therefore
$$\mu(\da^*)=\prod_\gotp\mu_\gotp(\dap^*)=\prod_\gotp\left(1-\frac{1}{\|\gotp\|}\right)=\frac{1}{\zeta_D(1)}\,.$$
In particular, $\mu(\da^*)=0$ if and only if $\zeta_D$ diverges in $1$.

In the case $D=\Z$, one can use this reasoning for the fun observation that Euler's celebrated proof of the existence of infinitely many prime numbers is equivalent to computing 
$\mu(\za^*)=0$. \end{eg} 

\subsection{Densities} We denote the power set of $D^n$ by $2^{D^n}$. 

\subsubsection{Definition of density} \label{sss:dns} By a density $d$ on $D^n$, we mean the datum of two functions 
$$d^+,\,d^-\colon 2^{D^n}\longrightarrow[0,1]$$ 
(called respectively upper and lower density) satisfying the properties (Dn1)-(Dn7) listed below. If $d^+(X)=d^-(X)$, we denote this common value by $d(X)$ and we say that $X$ has density $d(X)$. The requirements are the following (with $X$ and $Y$ varying among all subsets of $D^n$):\begin{itemize} 
\item[(Dn1)] $d^-(D^n)=1$\,;
\item[(Dn2)] $d^-(X)\le d^+(X)$\,;
\item[(Dn3)] $X\subseteq Y$ implies $d^*(X)\le d^*(Y)$, with $*\in\{+,-\}$\,;
\item[(Dn4)] if $D^n$ is the disjoint union of $X$ and $Y$, then $d^+(X)+d^-(Y)=1$;
\item[(Dn5)] if $X$ and $Y$ are disjoint, then 
\begin{equation}\label{e:d+} d^+(X\cup Y)\le d^+(X)+d^+(Y) \end{equation}
and
\begin{equation}\label{e:d-} d^-(X\cup Y)\ge d^-(X)+d^-(Y) \;;\end{equation}
\item[(Dn6)] for every $b\in D^n$, one has 
$$d^*(X+b)=d^*(X)\;,$$
with $*\in\{+,-\}$ and $X+b:=\{x+b\mid x\in X\}$;
\item[(Dn7)] for every ideal $\gota$ of $D$, one has
$$d^+(\gota D^n)=d^-(\gota D^n)=\frac{1}{\|\gota\|^n}\;.$$
\end{itemize} 

\begin{rmk} \label{r:cfltaxm} In the case $D=\Z$, the conditions above are almost equivalent to the axioms (F1)-(F5) proposed by Leonetti and Tringali in \cite{lt}. The difference is \eqref{e:d-}, which (as proved in \cite[Example 7]{lt}) is not a consequence of those axioms: we postulate it in order to have equality \eqref{e:d+-} below, which will play a crucial role in the proof of Theorem \ref{t:mt}. For a more precise comparison between our conditions and \cite{lt}, see \cite{dl3}. \end{rmk}

\begin{lem} Assume $d^+$ and $d^-$ satisfy conditions {\rm (Dn4)} and {\rm (Dn5)}. If $Y\subseteq D^n$ has a density, then the equality
\begin{equation}\label{e:d+-} d^-(X)=d(Y)-d^+(Y-X)\end{equation} 
holds for every $X\subseteq Y$. \end{lem}

\begin{prf} Let $X$ be a subset of $Y$ and $Z$ the complement of $Y$ in $D^n$. Then the disjoint union $Z\cup X$ is the complement of $Y-X$ and it follows
$$d^+(Y-X)\stackrel{\text{by (Dn4)}}{=}1-d^-(Z\cup X)\stackrel{\text{by \eqref{e:d-}}}{\le}1-d^-(Z)-d^-(X)=d^+(Y)-d^-(X)\,,$$
that is, $d^+(Y-X)+d^-(X)\le d^+(Y)$.\\
On the other hand, the complement of $X$ is $Z\cup(Y-X)$ and thus
$$d^-(X)\stackrel{\text{by (Dn4)}}{=}1-d^+(Z\cup(Y-X))\stackrel{\text{by \eqref{e:d+}}}{\ge}1-d^+(Z)-d^+(Y-X)=d^-(Y)-d^+(Y-X)\,,$$
that is, $d^+(Y-X)+d^-(X)\ge d^-(Y)$.\\
Therefore, if the density of $Y$ exists then one has $d(Y)=d^-(X)+d^+(Y-X)$. \end{prf} 

\subsubsection{Densities on $\da^n$} \label{sss:dnda} Let $d=(d^+,d^-)$ be a density on $D^n$. We extend the functions $d^+$ and $d^-$ (and hence the notion of density) to the power set of $\da^n$, by putting $d^*(X):=d^*(X\cap D^n)$ for any $X\subseteq\da^n$.

\begin{lem} \label{l:d=m0} Assume $d=(d^+,d^-)$ is a density on $D^n$. Then for every $X\subseteq\da^n$ and every $\gota\in\cali(D)$, the density of $X_\gota$ exists and it satisfies
\begin{equation} \label{e:d=m0} d(X_\gota)=\mu(X_\gota)=\frac{|\hpi_\gota(X)|}{\|\gota\|^n}\,. \end{equation} \end{lem}

\begin{prf} Recall that $X_\gota$ is a finite union of disjoint cosets of $\gota\da^n$. By (Dn7) and \eqref{e:hrdn}, we see that $d(\gota\da^n)$ exists and is equal to $\mu(\gota\da^n)$. The Haar measure is translation-invariant by definition and (Dn6) shows that so is the density. Finally, $\mu$ is additive on finite disjoint unions and (Dn5) implies that the same applies to $d$, if $d^+$ and $d^-$ coincide for each set in the union. \end{prf}

\begin{rmk} \label{r:dn8} Lemma \ref{l:d=m0} is the only reason why we postulated conditions (Dn6) and (Dn7). So they could be replaced by a request that \eqref{e:d=m0} holds for every $X$ and $\gota$. (Actually, if this happens then it is obvious that (Dn7) and at least a weak form of (Dn6) must be true.) \end{rmk}

\subsection{Examples of densities} \label{ss:egdns} We shall quickly discuss a few instances of pairs $(d^+,d^-)$ as in \S\ref{sss:dns}. For more (including generalizations to global Dedekind domains of the most commonly used densities over $\N$) see  \cite{dl3}.

\subsubsection{Buck density} \label{sss:Bd} Let $\cald_0\subset2^{D^n}$ consist of the subsets of the form $D^n\cap C$ where $C\subseteq\da^n$ is compact open: then $\cald_0$ is a subalgebra of $2^{D^n}$. By Lemma \ref{l:cmpopn}, one has 
$$\cald_0=\big\{\pi_\gota^{-1}(A)\mid\gota\in\cali(D)\text{ and }A\subseteq(D/\gota)^n\big\}\,.$$

Thus the density of $B\in\cald_0$ is fixed by Lemma \ref{l:d=m0} independently of the choice of $d$ (as long as (Dn1)-(Dn7) hold). We shall denote this value by $d(B)$. Following \cite{buck} (and simplifying the treatment along the lines of \cite[\S4]{lonied}), one can define the Buck density of $X\subseteq D^n$ by
$$\dsbkp(X):=\inf\big\{d(B)\mid B\in\cald_0\text{ and }X\subseteq B\big\}$$
and
$$\dsbkm(X):=\sup\big\{d(B)\mid B\in\cald_0\text{ and }X\supseteq B\big\}\,.$$  
Readers can check that $\dsbk$ satisfies conditions (Dn1)-(Dn7). More interestingly, Lemma \ref{l:d=m0} and formula \eqref{e:limidl} show that the equality
\begin{equation} \label{e:bckms} \dsbkp(X)=\mu(\xa) \end{equation}
is true for every $X\in2^{D^n}$.

Buck's original goal was to define an algebra of measurable subsets of $\N$. Following the approach in \cite{buck} (and generalized for example in \cite{lonied} and \cite{fptw}), let $X$ be a subset of $D^n$ and $Y=D^n-X$ its complement. We say that $X$ is {\em Buck-measurable} if $\dsbkp(X)+\dsbkp(Y)=1$; equivalently, $X$ is Buck-measurable if $\dsbkp(X)=\dsbkm(X)$. It is not hard to check that Buck-measurable sets form a subalgebra $\cald_{\rm Bk}$ of $2^{D^n}$ and that $\dsbkp$ is a finitely additive measure on $\cald_{\rm Bk}$; moreover, $\cald_{\rm Bk}$ is the Carath\'eodory closure of $\cald_0$ with respect to $\dsbkp$\,. \footnote{This means that it is the largest subalgebra $\cala\subseteq2^{D^n}$ such that for every $X\in\cala$ and every $\varepsilon>0$ one can find $A,B\in\cald_0$ satisfying $A\subseteq X\subseteq B$ and $\dsbkp(B-A)<\varepsilon$.}

\begin{rmks} \begin{itemize} \item[]
\item[{\bf 1.}] For $X\subseteq\N$, equality \eqref{e:bckms} appeared already in \cite[Theorem 1]{pst3}; this equality is also the main result in \cite{ok} (which mentions it as already 
known to Mauclaire, with a different proof in \cite{mcl6}).
\item[{\bf 2.}] The definition of Buck density works for any $D$ as in Assumption \ref{a:cntb} and can be extended to more general rings (see e.g.~\cite{fptw}). All other definitions 
of densities we are aware of assume that $D$ is a global Dedekind domain (with the set $S$ of places of $F$ not corresponding to prime ideals of $D$ playing a vital role in the constructions). \end{itemize} \end{rmks}

\subsubsection{Asymptotic densities} \label{sss:asnt} When $D$ is the ring of $S$-integers in a global field $F$, there is a natural embedding of it as a lattice in 
$$F_S:=\prod_{v\in S}F_v\,.$$
For every $v\in S$, let $|\cdot|_v$ be a normalized valuation on $F_v$ (in the sense of \cite[\S7 and \S11]{cas}): i.e., $|\cdot|_v$ is the standard absolute value if $F_v\simeq\R$, its square if $F_v\simeq\C$ and is defined by $|t|_v=q^{-1}$ if $F_v$ is non-archimedean, its residue field has $q$ elements and $t$ is a uniformizer. In the following, we shall say that $r\in F^*$ is sufficiently large if $\min_{v\in S}\{|r|_v\}\gg0$ and $r\rightarrow\infty$ will be a shortening for
\begin{equation} \label{e:rsffcbg} \min_{v\in S}\{|r|_v\}\rightarrow\infty\,. \end{equation} 

Let $B\subset F_S^n$ be any set with the following properties: \begin{itemize}
\item[(B1)] there is a positive constant $c\in\R$ such that, as,
\begin{equation} \label{e:dnsvlm} |D^n\cap rB|=c\prod_{v\in S}|r|_v^n+o\left(\prod_{v\in S}|r|_v^n\right) \end{equation}
where $r$ varies in $F^*$ and the small-$o$ notation refers to the growth of $r$,  in the sense of \eqref{e:rsffcbg};
\item[(B2)] for $r$ varying in $F^*$ and any fixed $b\in D^n$, the chain of inclusions
\begin{equation} \label{e:inclbx} r_1B\subseteq rB\cap(-b+rB)\subseteq rB\cup(-b+rB)\subseteq r_2B\,. \end{equation}
can be achieved with $r_1,r_2\in F^*$ which both satisfy $|r|_v-|r_i|_v=o(|r|_v)$ for every $v\in S$. \end{itemize}
Then the asymptotic density with respect to $B$ is defined by
\begin{equation} \label{e:d-asnt}
\dsbasp(X):=\limsup_{r\rightarrow\infty}\frac{|X\cap rB|}{|D^n\cap rB|}\;\;\text{ and }\;\;\dsbasm(X):=\liminf_{r\rightarrow\infty}\frac{|X\cap rB|}{|D^n\cap rB|}\;, \end{equation} 
where the limit is taken as $r$ varies in $F^*$.

\begin{prop} If {\em (B1)} and {\em (B2)} hold, the pair $(\dsbasp,\dsbasm)$ of \eqref{e:d-asnt} satisfies conditions {\em (Dn1)-(Dn7)}. \end{prop}

\begin{prf} It is straightforward to obtain (Dn1)-(Dn5) from \eqref{e:d-asnt}. As for the last two conditions, they follow from \eqref{e:dnsvlm} and respectively 
\begin{equation} \label{e:+b} |(b+X)\cap rB)|-|X\cap rB|=o\left(\prod_{v\in S}|r|_v^n\right) \end{equation}
and
\begin{equation} \label{e:k&rB} |\gota D^n\cap rB|=\frac{|D^n\cap rB|}{\|\gota\|^n}+o\left(\prod_{v\in S}|r|_v^n\right)\,. \end{equation}
In order to check \eqref{e:+b}, note that its left-hand side is bounded by the number of points of $D^n$ contained in the symmetric difference $rB\Delta(-b+rB)$. Together with \eqref{e:inclbx}, this yields the chain of inequalities
$$|(b+X)\cap rB)|-|X\cap rB|\le|D^n\cap(rB\Delta(-b+rB))|\le|D^n\cap r_2B|-|D^n\cap r_1B|\,,$$
where the magnitude of the last term is controlled by conditions (B1) and (B2).

Finally, if $\gota=aD$ is principal then \eqref{e:k&rB} is an immediate consequence of \eqref{e:dnsvlm} together with the obvious bijection between $aD^n\cap rB$ and $D^n\cap a^{-1}rB$. In the general case, just fix $a\in\gota$: then $\gota D^n$ can be written as a finite union of cosets of the form $aD^n+b$ and one obtains \eqref{e:k&rB} by applying \eqref{e:+b} to each of them. \end{prf}
  
\begin{rmk} Let $\mu_S$ be a Haar measure on $F_S^n$. The normalized valuations $|\cdot|_v$ are defined so that one has
$$\mu_S(rA)=\mu_S(A)\cdot\prod_{v\in S}|r|_v^n$$
for any measurable set $A\subseteq F_S^n$ and $r\in F^*$. Thus \eqref{e:dnsvlm} just says that the number of $D^n$-points in $B$ is proportional to its volume. \end{rmk}

Sets satisfying conditions (B1) and (B2) can be obtained by standard techniques of geometry of numbers. For example, fix a norm $\nu_v$ on $F_v^n$ for every $v\in S$ and let $B_v\subset F_v^n$ be the closed unit ball with respect $\nu_v$: then
$$B:=\prod_{v\in S}B_v\subset F_S^n$$  
enjoys the desired properties. We refer to \cite{dl3} for a detailed proof, as well as for alternative definitions of asymptotic densities. (Some special cases can be found in \cite[Proposition 2.11]{sch}.)

\subsubsection{Analytic density} \label{sss:anlt} Here we just consider the setting $n=1$.

We assume that $F$ is a global field and that $D^*$ is finite: by Dirichlet's unit theorem this happens exactly when $D$ is the ring of $S$-integers with $|S|=1$ (in the number field 
case, this means that $F$ is either $\Q$ or a quadratic imaginary extension of $\Q$). For $s$ positive real the series
$$\zeta_X^*(s):=\sum_{k\in X-\{0\}}\frac{1}{\|k\|^s}$$
is dominated by $|D^*|\cdot\zeta_F(s)$ (the Dedekind zeta function) and therefore it converges for every $s>1$. Moreover, a standard argument shows that $\zeta_D^*$ has a pole in $s=1$.

The analytic (or Dirichlet) density $\dsans$ is defined by taking
\begin{equation} \label{e:d-anlt} \dsans^+(X):=\limsup_{s\rightarrow1}\frac{\zeta_X^*(s)}{\zeta_D^*(s)}\;\;\text{ and }\;\;\dsans^-(X):=\liminf_{s\rightarrow1}\frac{\zeta_X^*(s)}{\zeta_D^*(s)}\,.  \end{equation}
It is not hard to check that conditions (Dn1)-(Dn7) are satisfied.

\begin{rmk} We will define a quite different notion of analytic density in Section \ref{s:anlt}. \end{rmk}

\subsection{Density and measure} \label{sss:dm} Throughout the rest of this section, we assume that we are given a density $d$ on $D^n$ in the sense of of \S\ref{sss:dns}: that is, there are two functions $d^+$, $d^-$ satisfying all the conditions (Dn1)-(Dn7).

\begin{lem} \label{l:e>d} Let $\calt$ be a subset of $\cals(D)$ having $0$ as an accumulation point and $X$ a subset of $\da^n$. Assume that for every $\sigma\in\calt$ the density of $X_\sigma$ exists and satisfies $d(X_\sigma)=\mu(X_\sigma)$. Then one has the inequality
\begin{equation} \label{e:m>d} d^+(X)\le\mu(\xa). \end{equation} \end{lem}

\begin{prf} The tautological inclusion $X\subseteq X_\sigma$ implies $d^+(X)\le d^+(X_\sigma)=d(X_\sigma)$ for every $\sigma\in\calt$ and hence
$$d^+(X)\le\limsup_{\sigma\rightarrow0\,,\,\sigma\in\calt}d(X_\sigma)=\limsup_{\sigma\rightarrow0\,,\,\sigma\in\calt}\mu(X_\sigma)=\mu(\xa)$$
where the last equality follows from Lemma \ref{l:msrlmt}. \end{prf}

\begin{cor} \label{c:m>d} The inequality \eqref{e:m>d} holds for every $X\subseteq\da^n$. \end{cor}

\begin{prf} By Remark \ref{r:idl}.1 we can take $\calt=\cali(D)$. Then Lemma \ref{l:d=m0} ensures that Lemma \ref{l:e>d} can be applied for every $X\subseteq\da^n$. \end{prf}

It might be worth stating explicitly the following obvious consequence of Corollary \ref{c:m>d}.

\begin{cor} \label{c:m=0} For any $X\subseteq\da^n$, if $\mu(\xa)=0$ then also $d(X)=0$. \end{cor}

\begin{rmks} \label{r:m>d} Versions of inequality \eqref{e:m>d} and Corollary \ref{c:m>d} have been discovered many times - see for example \cite[Theorem 3]{nov0}, \cite[Remark (iv) on page 201]{lonied}, \cite[Lemma 1.1]{ekd} and \cite[Theorem 3]{lt}. \begin{enumerate}
\item[{\bf 1.}] The inequality proved by Ekedahl in \cite[Lemma 1.1]{ekd} can be written as $d(\Sigma)\le\mu(\Sigma)$, where $\Sigma\subseteq\za^n$ is any measurable set (and $d$ an asymptotic density). However, the condition of measurability is too weak: for example, $\N$ is a measurable subset of $\za$, with $\mu(\N)=0$ (since it is countable) and its density is 1 (counting as in \cite{ekd}). The first statement of Ekedahl's proof shows that the number $\displaystyle A=\lim_{M\rightarrow\infty}c_M/(M!)^n$ given as $\mu(\Sigma)$ is really $\mu(\widehat\Sigma)$.
\item[{\bf 2.}] As shown in \cite[Theorem 3]{lt}, the proof of Corollary \ref{c:m>d} does not really require the full strength of properties (Dn1)-(Dn7); actually, it is enough that the 
Leonetti-Tringali axioms (F1), (F3) and (F6) hold. In \cite{lt}, pairs $(d^+,d^-)$ satisfying these weaker conditions are called quasi-densities. Corollary \ref{c:m=0} can be strengthened to ``quasi-densities'' - see \cite[Theorem 2.4]{lt2}.  \end{enumerate} \end{rmks}

Let $\partial Z$ denote the boundary of $Z\subseteq\da^n$.

\begin{lem} \label{l:brtt} Let $X$ be a subset of $\da^n$ and assume $\mu(\partial\xa)=0$. Then the density of the set $\xa$ exists and it is equal to $\mu(\xa)$. \end{lem}

\begin{prf} Put $W=\da^n-\xa$, so to have $\partial\xa=\xa\cap\widehat W$. The obvious equality 
$$\mu(\xa)+\mu(\widehat W)=\mu(\xa\cup\widehat W)+\mu(\xa\cap\widehat W)$$ 
and the hypothesis $\mu(\xa\cap\widehat W)=0$ imply $\mu(\xa)+\mu(\widehat W)=1.$ Moreover, $d^-(\xa)+d^+(W)=1$, by (Dn4). The claim follows from the chain of inequalities
$$\mu(\xa)\stackrel{\text{by \eqref{e:m>d}}}{\ge}d^+(\xa)\ge d^-(\xa)=1-d^+(W)\stackrel{\text{by \eqref{e:m>d}}}{\ge}1-\mu(\widehat W)=\mu(\xa)\,.$$
\end{prf}

\begin{prop} \label{p:intrs} Let $X$ be a subset of $D^n$ and put $Y=D^n-X$. Assume $\mu(\xa\cap\widehat Y)=0$. Then the density of the set $X$ exists and it is equal to $\mu(\xa)$. \end{prop}

\begin{prf} Let $W$ be as in the proof of Lemma \ref{l:brtt} and put $V=W\cap D^n$. We have $\widehat W=\widehat V$ because $W$ is open and $D^n$ is dense; moreover $V\subseteq Y$ yields $\widehat W\subseteq\widehat Y$. Hence $\mu(\partial\xa)=0$ and $\xa$ has density $\mu(\xa)$, by Lemma \ref{l:brtt}. By Corollary \ref{c:m=0}, the inclusion $\xa\cap Y\subseteq\xa\cap\widehat Y$ yields $d(\xa\cap Y)=0$. Therefore 
$$d(\xa)=d(\xa\cap D^n)\stackrel{\text{by \eqref{e:d+-}}}{=}d(\xa\cap Y)+d(X)=d(X)\,.$$
\end{prf}

\begin{rmk} \label{r:bmsr} As above, assume $D^n=X\sqcup Y$. Then $\da^n=\xa\cup\widehat Y$ and hence we have
\begin{equation} \label{e:bckbnl} \mu(\xa)+\mu(\widehat Y)=\mu(\xa\cup\widehat Y)+\mu(\xa\cap\widehat Y)=1+\mu(\xa\cap\widehat Y). \end{equation} 
Remembering \S\ref{sss:Bd}, it is straightforward to see from \eqref{e:bckbnl} that $X$ is Buck-measurable if and only if $\mu(\xa\cap\widehat Y)=0$. Therefore Proposition \ref{p:intrs} proves that the density exists for every Buck-measurable set, whatever choice of $d$ has been made, as long as (Dn1)-(Dn7) hold. One can easily check that the Buck algebra $\cald_{\rm Bk}$ is the maximal subalgebra of $2^{D^n}$ such that for any of its elements the density exists and it equals the Haar measure of the closure. 

Furthermore the inclusion $\partial\xa\subseteq\xa\cap\widehat Y$ (which becomes an equality if $X$ is closed, i.e., if $X=\xa\cap D^n$) shows that the condition $\mu(\partial\xa)=0$ is necessary (and sufficient, if $X$ is closed) in order to have $X$ Buck-measurable. This explains why Buck's theory cannot be used to study the density, for example, of the set $X$ of square-free integers in $\Z$ (or in $\N$), as we mentioned in the introduction. Such an $X$ is indeed closed in $\Z$ (respectively $\N$), but $\xa$ has empty interior (as we shall see as a consequence of Corollary \ref{c:kpfree}), so that $\partial\xa=\xa$; and one finds $\mu(\xa)>0$. \end{rmk}

\begin{thm} \label{t:mt} Let $\calt$ be a subset of $\cals(D)$ having $0$ as an accumulation point and $X$ a subset of $D^n$. Assume that for every $\sigma\in\calt$ the density of $X_\sigma$ exists and satisfies $d(X_\sigma)=\mu(X_\sigma)$. Then the equality
\begin{equation} \label{e:banale} \lim_{\sigma\rightarrow0} d^+(X_\sigma-X)=0 \end{equation}
holds if and only if $d(X)$ exists and it satisfies $d(X)=\mu(\xa)$. \end{thm}

\begin{prf} By Lemma \ref{l:e>d}, the inequality \eqref{e:m>d} holds and we just need to prove that \eqref{e:banale} is equivalent to $d^-(X)=\mu(\xa)$. For this, it is enough to note that
$$d^-(X)\stackrel{\text{by \eqref{e:d+-}}}{=}d(X_\sigma)-d^+(X_\sigma-X)=\mu(X_\sigma)-d^+(X_\sigma-X)$$
is true for every $\sigma\in\calt$. Therefore \eqref{e:limmu} implies
$$\lim_{\sigma\rightarrow0}d^+(X_\sigma-X)=\lim_{\sigma\rightarrow0}\big(\mu(X_\sigma)-d^-(X)\big)=\mu(\xa)-d^-(X)\,.$$
\end{prf}

\subsubsection{Close pairs of sets} \label{sss:clssts} Let $X\Delta Y$ denote the symmetric difference of $X,Y\in2^{D^n}$. The following is essentially a restatement of \cite[Proposition 1(i)]{lt}; we include a proof for the convenience of the reader.

\begin{lem} \label{l:vc} Let $d$ be a density on $D^n$. If $d^+(X\Delta Y)=0$ then $d^+(X)=d^+(Y)$ and $d^-(X)=d^-(Y)$. \end{lem}

\begin{prf} By $X\cup Y=(X\cap Y)\sqcup(X\Delta Y)$ and (Dn3), (Dn5) one gets
$$d^+(X\cap Y)\le d^+(X\cup Y)\le d^+(X\cap Y)+d^+(X\Delta Y)\,.$$
Since $X$ and $Y$ are both intermediate between $X\cap Y$ and $X\cup Y$, it follows that $d^+(X\Delta Y)=0$ implies $d^+(X)=d^+(Y)$. As for the second equality, use 
$$(D^n-X)\Delta(D^n-Y)=X\Delta Y$$
and $d^-(X)=d^+(D^n-X)$. \end{prf}

Let $X\widehat\Delta Y$ denote the closure of the symmetric difference $X\Delta Y$. We say that $X$ and $Y$ are {\em close} if $\mu(X\widehat\Delta Y)=0$. 

\begin{prop} \label{p:clssts} If two subsets $X$, $Y$ of $D^n$ are close, then  the equalities $d^{\pm}(X)=d^{\pm}(Y)$ hold for any choice of $d$. Moreover, $\mu(\xa)=\mu(\widehat Y)$. \end{prop}

\begin{prf} The first statement is immediate from Corollary \ref{c:m=0} and Lemma \ref{l:vc}. As for the second one, it follows from the inclusion
$$\xa\Delta\widehat Y\subseteq(X\widehat\Delta Y)\,$$
which is proved by an elementary reasoning.  \end{prf}

We say that $X\subseteq D^n$ is {\em almost closed} if it is close to its closure $\xa\cap D^n$.

\begin{cor} \label{c:bmsr} A subset $X$ of $D^n$ is Buck-measurable if and only if it is almost closed and $\mu(\partial\xa)=0$. \end{cor}

\begin{prf} Put $Z:=\xa\cap D^n$ and $W:=X\Delta Z=Z-X$. 

Sufficiency: since $Z$ is closed, it follows by Remark \ref{r:bmsr} that it is Buck-measurable if $\mu(\partial\widehat Z)=0$, which holds because $\widehat Z=\xa$, and then Proposition 
\ref{p:clssts} implies that any set close to $Z$ is also Buck-measurable.
 
Necessity: it was proved in Remark \ref{r:bmsr} that if $X$ is Buck-measurable one must have $\mu(\partial\xa)=0$. As for being close to its closure, \eqref{e:bckms} yields
$$\dsbkp(X)=\mu(\xa)=\mu(\widehat Z)=\dsbk(Z)$$ 
and hence, by \eqref{e:d+-},
$$\dsbk(Z)=\dsbkm(X)+\dsbkp(W)=\dsbkm(X)+\mu(\widehat W)\,,$$
proving that $X$ is not Buck-measurable if $\mu(\widehat W)>0$. \end{prf}

\subsubsection{The role of closed sets} \label{sss:rlclsd} Even if not strictly necessary, when comparing $d(X)$ and $\mu(\xa)$ it is convenient to assume that $X$ is a closed (or at least almost closed) subset of $D^n$. On the one hand, if $d(X)=\mu(\xa)$ is true for $X$ then it is true also for its closure. This is immediate from the chain of inequalities
$$\mu(\xa)\ge d^+(\xa)\ge d^-(\xa)\ge d^-(X)\,,$$
which holds for every $X\subseteq D^n$ because of (Dn2), (Dn3) and Corollary \ref{c:m>d} and shows that $d(X)=\mu(\xa)$ implies $d(X)=d(\xa\cap D^n)$. 

On the other hand, if $X$ is far from being closed then there is no reason to expect equality between density and measure. In particular, if $X$ is dense in $\da^n$ the density can be 
anything. We give a few examples with $D=\Z$ and $n=1$. \begin{enumerate}
\item Let $Y$ be the set of square-free integers and $X$ its complement. It is not hard to see that $X$ is dense in $\za$ (we will prove later a slightly more general result - see Corollary \ref{c:kpfree}). The equality $\dsas(X)=1-\frac{6}{\pi^2}$ is a classical result.
\item Let $X$ be the set of all positive integers having first digit $1$ in their decimal expansion. Then $\xa=\za$, so $\mu(\xa)=1$, while $X$ has no asymptotic density and its logarithmic density exists and is strictly smaller than 1 (more precisely, in $\N$ one has $\dsasm(X)=\frac{1}{9}$, $\dsasp(X)=\frac{5}{9}$ and $\dlib_{\log}(X)=\frac{\log2}{\log10}$, as explained in \cite[pages 415 and 417]{tnm}).
\item It is easy to find a dense set $X$ with $\dsas(X)=0$. The simplest and most classical example, often mentioned in the literature, is probably $X=\{n!+n\mid n\in\N\}$. \end{enumerate}

It would be interesting to have an example of a closed $X\subseteq D^n$ and a density $d$ such that $d(X)$ exists and it is different from $\mu(\xa)$ (equivalently, the limit in \eqref{e:banale} is not $0$). By results of Davenport-Erd\H os \cite{de1} and Besicovitch \cite{bsc}, one can obtain an example of a closed $X\subseteq\Z$ such that $\mu(\xa)=\dlib_{\log}(X)=\dsasp(X)>\dsasm(X)$. See Remark \ref{r:deN}.1 below.

\subsubsection{Some counterexamples} \label{sss:cnt} 
 
The first one implies that \cite[Proposition 2.2]{ekd} is wrong. Since the latter plays a key role in the proof of \cite[Theorem 2.3]{ekd}, one should expect that also this second claim is false, as we will show. \\

\noindent 1. {\em There is no need for an open subset of $\da$ to have a density equal to its measure}. By Assumption \ref{a:cntb}, we can write $D=\{x_n\}_{n\in\N}$. For a given $\varepsilon$ in the open interval $(0,1)$, choose $a\in D$ such that 
$$\sum_{n=1}^\infty\frac{1}{|D/aD|^n}<\varepsilon$$
and, for every $n\in\N-\{0\}$, let $U_n$ be the coset $x_n+a^n\da$. Put $U=\bigcup_nU_n$. Then $U$ is an open subset of $\da$ with measure
$$\mu(U)\le\sum\mu(U_n)=\sum\frac{1}{|D/aD|^n}<\varepsilon\;.$$
On the other hand, $d(U)=d(U\cap D)=d(D)=1$ for every choice of $d$.\\ 

\noindent 2. {\em A counterexample to \cite[Theorem 2.3]{ekd}.} Let $Y$ be the closed subscheme of the affine plane on $\Z$ corresponding to the point $(0,0)$. Then the set
$$W:=\bigcap_{p\in\calp(\Z)-\{2\}}\hpi_{p\Z}^{-1}(\F_p^2-Y(\F_p))$$
consists of all pairs $(a,b)\in\za^2$ having greatest common divisor a power of $2$ (including $2^0$). It is well-known (and can be proved for example by \cite[Theorem 1.2]{ekd}) that the asymptotic density of $W$ exists, with value
$$\dsas(W)=\prod_{p\neq2}\left(1-\frac{1}{p^2}\right)=\left(1-\frac{1}{4}\right)\cdot\frac{1}{\zeta(2)}=\frac{8}{\pi^2}\,.$$
Then, according to \cite[Theorem 2.3]{ekd}, for any open subset $T$ of $\Z_2^2$ of Haar measure $t$ the set
$$X_T:=W\cap\hpi_{2^\infty}^{-1}(T)$$
should have density $\frac{8}{\pi^2}t$. However one can reason as in counterexample 1 and take $T$ containing $\Z^2$ with $t<1$ to obtain a contradiction, since for such a $T$ one has
$$X_T\cap\Z^2=W\cap\Z^2$$
and hence $\dsas(X_T)=\frac{8}{\pi^2}>\frac{8}{\pi^2}t$. 

With a little more effort, a construction similar to the one used in counterexample 3 below can be used to show that it is also possible to choose $T$ so that $X_T$ has no 
asymptotic density.\\

\noindent 3. {\em There is no reason to expect that open or closed subsets of $D^n$ have a density}. We give an example (based of an idea of Bary-Soroker) with $d=\dsan$\,, the analytic 
density in $\N$ (since the existence of $\dsan$ implies that the asymptotic density exists as well - see e.g. \cite[Theorems III.1.2 and III.1.3]{tnm}; we work in $\N$ rather than $\Z$ because it simplifies notations without affecting the argument). 

Let $A\subseteq\N$ be any set which has no analytic density (that  is, the ratio $\zeta_A(s)/\zeta(s)$ diverges in $1$, with $\zeta_A(s):=\sum_{a\in A}a^{-s}$), so that we can find $m\in\N$ satisfying
$$\frac{1}{m-1}<\dsanp(A)-\dsanm(A)\,.$$
For simplicity we assume $0\notin A$. Then
$$U=\{a+jm^a\mid a\in A, j\in\N\}=\N\cap\bigcup_{a\in A}(a+m^a\Z)$$
is a proper open subset of $\N$ (in the topology induced by the dense embedding $\N\hookrightarrow\za$). One computes
\begin{eqnarray*} \zeta_U(s):=\sum_{k\in U}\frac{1}{k^s} & = & \sum_{a\in A}\sum_{j=0}^\infty\frac{1}{(a+jm^a)^s}\le\zeta_A(s)+\sum_{j=1}^\infty\sum_{a\in A}\frac{1}{(jm^a)^s} \\
& \le & \zeta_A(s)+\sum_{j=1}^\infty\frac{1}{j^s}\sum_{n=1}^\infty\frac{1}{(m^s)^n}=\zeta_A(s)+\frac{1}{m^s-1}\,\zeta(s)\,. \end{eqnarray*}
Dividing by $\zeta(s)$ and taking the limit as $s$ goes to $1$ we obtain
$$\dsanm(U)\le\dsanm(A)+\frac{1}{m-1}<\dsanp(A)\le\dsanp(U)\,,$$
proving that $U$ has no density. Therefore the closed set $X=\N-U$ has no density too.

\section{Densities on $\cals(D)$ and the Davenport-Erd\H os Theorem} \label{s:anlt}

In this section we assume that $D$ is global (that is, it is the ring of $S$-integers in a global field).  

\subsection{Densities on the set of ideals} The general definition of density in \S\ref{sss:dns} has a straightforward adaptation to the set of ideals of $D$. Namely, a density $d$ on $\cali(D)$ is the datum of two functions 
$$d^+,\,d^-\colon 2^{\cali(D)}\longrightarrow[0,1]$$  
satisfying the following properties: \begin{itemize}
\item[(Dn$_\cali$1)-(Dn$_\cali$5)] same as (Dn1)-(Dn5), just replacing subsets of $D^n$  with subsets of $\cali(D)$;
\item[(Dn$_\cali$7)]  for every ideal $\gota$ of $D$, one has
$$d^+\big(\gota\cali(D)\big)=d^-\big(\gota\cali(D))=\frac{1}{\|\gota\|}\;,$$
where $\gota\cali(D):=\{\gotb\in\cali(D):\gotb\subseteq\gota\}$. \end{itemize}
We do not ask for an analogue of (Dn6), for the obvious reason that the additive  group $D$ does not act on the multiplicative monoid $\cali(D)$.

We saw in \S\ref{sss:sprn} that one can think of $\cali(D)$ as a dense subset of the topological space $\cals(D)$. As before, we keep the convention of using German letters for ideals in $\cali(D)$ and Greek letters for generic supernatural ideals. As in \S\ref{sss:dnda}, any density $d$ on $\cali(D)$ is extended to $\cals(D)$ by $d^*(\calx):=d^*(\calx\cap\cali(D))$.

\subsubsection{The map $\rho$ and push-forward densities} Recall the map $\rho\colon\da\longrightarrow\cals(D)$ defined in \eqref{e:ro}.

\begin{lem} \label{l:ropn} The map $\rho$ is open. \end{lem}

\begin{prf} For any $x\in\da$ one has $x\da=\{z\in\da\mid v_\gotp(z)\ge v_\gotp(x)\;\forall\,\gotp\in\calp(D)\}\,.$ Hence
$$\rho(x)=\rho(y)\Longleftrightarrow x\da=y\da\Longleftrightarrow v_\gotp(x)=v_\gotp(y)\;\forall\,\gotp\in\calp(D)\Longleftrightarrow\exists\,u\in\da^*:x=uy\,.$$
(For the last implication, use the fact that if $v_\gotp(x)=v_\gotp(y)=\infty$ then the components of $x$ and $y$ in $\gotp$ are both $0$.) Thus for any $X\subseteq\da$ we get
\begin{equation} \label{e:bln} \rho^{-1}(\rho(X))=\da^*X:=\{ux\mid u\in\da^*,\,x\in X\}\,. \end{equation}
Multiplication by a unit induces a homeomorphism of $\da$ in itself: hence if $X$ is open then so is $\da^*X=\cup_{u\in\da^*}uX$.
\end{prf}

\begin{rmk} \label{r:daqt} The proof of Lemma \ref{l:ropn} also shows that $\rho$ identifies $\cals(D)$ with $\da/\da^*$. \end{rmk}

Any density $d=(d^+,d^-)$ on $D$ induces a density $\rho_* d=(\rho_* d^+,\rho_* d^-)$ on $\cali(D)$ by
$$\rho_* d^\pm(\calx):=d^\pm(\rho^{-1}(\calx))\,.$$ 
It is immediate  to check that if (Dn1)-(Dn5) hold for $d$ then (Dn$_\cali$1)-(Dn$_\cali$5) are inherited by $\rho_* d$. The same applies to (Dn$_\cali$7), because
$$\gotb\in\gota\cali(D)\Longleftrightarrow\gotb\subseteq\gota\Longleftrightarrow\widehat\gotb\subseteq\widehat\gota\Longleftrightarrow\widehat\gotb=a\da\text{ for some }a\in\widehat\gota$$
(by Lemma \ref{l:clpr}) proves the equality
\begin{equation} \label{e:idmlt} \gota\cali(D)=\rho(\widehat\gota)\cap\cali(D) \end{equation}
and Remark \ref{r:daqt} implies $\rho^{-1}(\rho(\widehat\gota))=\widehat\gota$, yielding $\rho_* d(\gota\cali(D))=d(\widehat\gota)=d(\gota)=\mu(\widehat\gota)$ by (Dn7).

\begin{rmk} \label{r:ztXvY} The set $\gota\cali(D)$ is contained in $\rho(\widehat\gota)$, but not in $\rho(\gota)$ (unless $D$ has class number one). It is important to note that, because of the existence of ideals which are not principal, when working with densities on $\cali(D)$ we should expect $d^\pm(\rho(X))\neq d^\pm(\rho(\xa))$ even when $X$ is closed in $D$ (that is, $X=D\cap\xa$).  \end{rmk}

\subsubsection{Analytic density for a  general $D$} \label{sss:anltid} As implicit in \S\ref{sss:anlt}, the analytic density on $\N$ does not have an obvious generalization to $D$ when there are infinitely many units: this is because the zeta function is an efficient tool to count ideals rather than elements. On the other hand, we are going to see that the construction in \S\ref{sss:anlt} works nicely as a density on $\cali(D)$.

For any $\calx\subseteq\cals(D)$, define 
\begin{equation} \label{e:ztX} \zeta_\calx(s):=\sum_{\gotn\in\calx\cap\cali(D)}\frac{1}{\|\gotn\|^s}\;. \end{equation}
In particular, if $D$ is the ring of integers of a number field then $\zeta_{\cali(D)}$ is the same as the usual Dedekind zeta function. Our hypothesis that $D$ is a global Dedekind 
domain ensures that $\zeta_\calx(s)$ always converges, absolutely and uniformly on compact subsets of the complex half-plane ${\rm Re}(s)>1$, for every $\calx$. 

We shall use the shortenings $\zeta_D$ and $\zeta_\gota$ for $\zeta_{\cali(D)}$ and $\zeta_{\gota\cali(D)}$. The ratio 
\begin{equation} \label{e:dltXs} \delta(\calx,s):=\frac{\zeta_\calx(s)}{\zeta_D(s)} \end{equation}
takes values in $[0,1]$ for every $s$ in the real half-line $(1,\infty$). The {\em analytic density} $\dsan$ is defined by
\begin{equation} \label{e:d-anltid}
\dsanp(\calx):=\limsup_{s\rightarrow1}\delta(\calx,s)\;\;\text{ and }\;\;\dsanm(\calx):=\liminf_{s\rightarrow1}\delta(\calx,s)\,. \end{equation}

\begin{eg} \label{eg:dnl} In the case $D=\Z$, the set of ideals is canonically identified with $\N$. Therefore the analytic density \eqref{e:d-anltid} coincides with  the traditional definition (as found for example in \cite[Chapter III.1.3]{tnm}). We also note that one has $\dsan=\rho_*\dsans$ when $D^*$ is finite (the setting of \S\ref{sss:anlt}) and $D$ has class number 1. However, if $D$ is just a global Dedekind domain, we don't know how to define a density on it having $\dsan$ as its push-forward and we see no reason to expect that such a density should exist. \end{eg}

\begin{prop} \label{p:andid} The analytic density $\dsan$ satisfies conditions {\rm (Dn$_\cali$1)-(Dn$_\cali$5)} and {\rm (Dn$_\cali$7)}. \end{prop}

\begin{prf} We just check (Dn$_\cali$7), since the rest is trivial. In order to prove that the analytic density exists and it satisfies $\dsan(\gota\cali(D))=\mu(\widehat\gota)$ for any 
$\gota\in\cali(D)$, it is enough to show 
$$\zeta_\gota(s)=\frac{1}{\|\gota\|^s}\,\zeta_D(s)\;\text{ for every }s\in(1,\infty).$$
This is immediate from \eqref{e:ztX}, since
$$\gotn\in\gota\cali(D)\Longleftrightarrow \gotn\subseteq\gota \Longleftrightarrow \gotn=\gota\gotn_1 \Longleftrightarrow\|\gotn\|^s=\|\gota\|^s\cdot\|\gotn_1\|^s\,. $$
\end{prf}

\subsubsection{The measure on $\cals(D)$} \label{sss:SDmsr} The Haar measure $\mu$ on $\da$ is pushed forward by $\rho$ to a measure $\rho_*\mu$ on $\cals(D)$, by
\begin{equation} \label{e:romsr} \rho_*\mu(\calx):=\mu(\rho^{-1}(\calx))\,. \end{equation}
Points in $\cals(D)$ have measure zero with respect to $\rho_*\mu$, as follows from Example \ref{eg:da*0} and formula \eqref{e:msrprd} together with \eqref{e:bln}.

\begin{rmk} Note that $\rho_*\mu$ is not a Haar measure on the monoid $\cals(D)$, in the sense that it is not invariant under the operation, as one immediately sees from 
$$\rho_*\mu\big(\sigma\cdot\cals(D)\big)=\rho_*\mu(\rho(\sigma\cdot\da))=\mu(\sigma)=\frac{1}{\|\sigma\|}\neq\rho_*\mu\big(\cals(D)\big)=1\,.$$
In high-brow terms, one could say that this is because $\mu$ (and hence $\rho_*\mu$) comes from the measure on the additive adele group attached to $D$, while a product  invariant measure must be related to the group of ideles. \end{rmk}

\subsection{``Sets of multiples''} Classically the ``set of multiples'' of a (possiby finite) $\N$-valued sequence $(a_i)$ is the set of all the positive integers divisible by some $a_i$: that is, a set of multiples consists of the intersection of $\N-\{0\}$ with a family of ideals of $\Z$. Note that sets of multiples are open in the topology induced on $\N$ by $\za$. Hence the natural analogue in our setting is the following. Let $(\gota_i)_{i\in\N}$ be a sequence of ideals of $D$ and put 
\begin{equation} \label{e:dfU} U_n:=\bigcup_{i\le n}\widehat\gota_i\;\;\,,\;\;U:=\bigcup_{i\in\N}\widehat\gota_i\,. \end{equation}
The set of multiples of $(\gota_i)$ is $U\cap D$. 

\subsubsection{Primitive sequences} In the study of sets of multiples, traditionally one works just with primitive sequences, that is, sequences $(a_i)$ in $\N$ such that $a_i|a_j$ only 
if $i=j$. It is straightforward to check that the sequence $(a_i)$ is primitive if and only if there is no $i$ satisfying $a_i\Z\subseteq\cup_{j\neq i}a_j\Z$. Therefore the requirement that a sequence is primitive eliminates redundancy in constructing the set of multiples.

\begin{lem} \label{l:prmnc} Assume that the class group of $D$ is not trivial. Then there is an ideal $\gota_0$ and a set $\calj\subset\cali(D)$ such that
\begin{equation} \label{e:prmnc} \gota_0\subseteq\bigcup_{\gotb\in\calj}\gotb \end{equation}
and no divisibility relation holds between any two distinct elements of $\calj\cup\{\gota_0\}$. \end{lem}

\begin{prf} Choose $\gotp,\gotq\in\calp(D)$ with $\gotp$ principal and $\gotq$ not (this is always possible thanks to Chebotarev). Then one can take $\gota_0=\gotp\gotq$ and 
$$\calj=\{\gotp^2,\gotq^2\}\cup\big(\calp(D)-\{\gotp,\gotq\}\big)\,.$$
Indeed, any two elements in $\calj$ are coprime and none of them divides $\gota_0$, nor can be divided by it. As for \eqref{e:prmnc}, we show that every $x\in\gota_0$ belongs to an ideal in $\calj$. Since $\gota_0$ is not principal, one has $xD=\gotp\gotq\gotc$ for some proper ideal $\gotc$. If $\gotr$ is any prime which divides $\gotc$, then one of $\gotr,\gotp\gotr,\gotq\gotr$ is in $\calj$ and all of them contain $x$. \end{prf}

We say that a sequence of ideals $(\gota_i)$ in $\cali(D)$ is {\em primitive} if there is no index $i$ such that $\gota_i\subseteq\cup_{j\neq i}\gota_j\,.$ Lemma \ref{l:prmnc} proves that 
primitivity cannot be expressed by a divisibility condition on ideals of $D$ when the class group is not trivial.

\begin{rmk} \label{r:prmnc} Since every closed ideal in $\da$ is principal, the inclusion and divisibility conditions coincide for completed ideals:
$$\widehat\gota_i\subseteq\cup_{j\neq i}\widehat\gota_j\Longleftrightarrow\exists\,j\neq i\text{ such that }\widehat\gota_j|\widehat\gota_i\Longleftrightarrow\exists\,j\neq i\text{ such that }\gota_j|\gota_i\,.$$
This also shows that  $\widehat\gota_i\subseteq\cup_{j\neq i}\widehat\gota_j$ implies  $\gota_i\subseteq\cup_{j\neq i}\gota_j$, while there is no reason for the converse to be true if the sequence is infinite: just take $\gota_0$ and $\calj$ as in Lemma \ref{l:prmnc} to get  an example where $\widehat\gota_0\nsubseteq\cup\widehat\gotb$ in spite of \eqref{e:prmnc}. \end{rmk}

\subsubsection{An application} The fact that ``sets of multiples'' are  open allows to give an easy answer to a question asked in \cite{lt2}. Recall the function 
$\omega\colon\cals(D)\rightarrow\bar\N$ defined in \eqref{e:ommin}.

\begin{prop} \label{p:msromg} Fix $k\in\N$ and put $Z_k:=\{x\in D\mid\omega(xD)=k\}$. Then $\mu(\widehat Z_k)=0$. \end{prop}

\begin{prf} Consider $\calt:=\{\gota\in\cali(D)\mid\omega(\gota)>k\}\,.$ Then $U:=\cup_{\gota\in\calt}\widehat\gota$ is open and hence so is its image $\rho(U)$ in $\cals(D)$. Thus 
$$\calk:=\cals(D)-\rho(U)$$
is a closed set. Note that $\sigma\in\rho(U)$ implies $\omega(\sigma)>k$, by the trivial inequality
$$\omega(\widehat\gota)\le\omega(\rho(a))\;\,\forall\,a\in\widehat\gota\,.$$
On the other hand assume $\omega(\sigma)>k$ with $\sigma=x\da$. Then $x\in\widehat\gotp_i$ for distinct primes $\gotp_1$, ..., $\gotp_{k+1}$. Since the ideal $\gota=\prod_{i=1}^{k+1}\gotp_i$ is in $\calt$, we obtain 
$$\sigma=\rho(x)\in\rho(\widehat\gota)\subseteq\rho(U)\,.$$ This proves the equality  
\begin{equation} \label{e:kappa} \calk=\{\sigma\in\cals(D)\mid\omega(\sigma)\le k\}\,. \end{equation}
From \eqref{e:kappa} it is easy to see that $\calk$ is countable. Thus $\rho^{-1}(\calk)$ is a countable union of subsets which, by \eqref{e:bln}, all have the form $a\da^*$. Therefore, by 
\eqref{e:msrprd} and Example \ref{eg:da*0}, we get $\mu(\rho^{-1}(\calk))=0.$ The proof is concluded by the simple observation that $\widehat Z_k$ is a subset of $\rho^{-1}(\calk)$, since 
$\rho(Z_k)\subseteq\calk$ by \eqref{e:kappa} and $\rho^{-1}(\calk)$ is closed. \end{prf}

\begin{cor} \label{c:lt2} The set $\{x\in D\mid\omega(xD)=k\}$ has density $0$ for every $k\in\N$ and every definition of density satisfying {\rm (Dn1)-(Dn7)} or, more generally, the 
Leonetti-Tringali axioms {\rm (F1), (F3)} and {\rm (F6)}. \end{cor}

\begin{prf} Just apply Proposition \ref{p:msromg} together with Corollary \ref{c:m>d} and Remark \ref{r:m>d}.2. \end{prf}

In the case $D=\Z$, this strengthens \cite[Theorem 3]{bf} (which proves that each $Z_k$ has uniform density $0$) and provides an affirmative answer to \cite[Question 5.2]{lt2}. All the main ideas for this special case were discussed in a S.U.R.F. project at XJTLU in summer 2016 and already appeared in \cite{saett}. Our result also shows that composite integers have density $1$ in any global Dedekind domain, since, by definition, $x$ composite means $\omega(x)>1$. This improves on \cite[Example 4.6]{lonied}, where it was observed that the upper Buck density of the set of all composite algebraic integers in a number field is $1$.

\begin{rmk} Corollary \ref{c:lt2} can be interpreted as stating that the probability that an integer $n$ is divisible by $k$ different primes is $0$. Taking the sum over all $k\in\N$, one obtains the apparently paradoxical result
\begin{equation} \label{e:prbbdvf} {\rm Prob}\big(\omega(n)<\infty\big)\stackrel{?}{=}\sum_{k\in\N}{\rm Prob}\big(\omega(n)=k)=0\,. \end{equation}
The definition of probability implicit on the right-hand side of \eqref{e:prbbdvf} is
$${\rm Prob}\big(x\in X\big):=\mu(\xa)$$
and this makes clear where the mistake lies: the dubious equality in \eqref{e:prbbdvf} is false because the closure of a union is not the union of the closures. The explanation of the paradox is completed observing that the set $\{z\in\za\mid\omega(z\za)=\infty\}$ has measure $1$, but its complement is dense (it contains the non-negative integers). \end{rmk}

\subsubsection{The closure of $U$} For $\gota\in\cali(D)$, we will write $\supp(\gota)$ to denote its support (that is, the set of primes dividing $\gota$). 

\begin{lem} \label{l:Udns} Take $U=\cup_i\widehat\gota_i$ as in \eqref{e:dfU}. Then $U$ is dense in $\da$ if and only if for every finite $S\subset\calp(D)$ there is an index $i$ such that 
$\supp(\gota_i)\cap S=\emptyset$. \end{lem}

\begin{prf} We are going to prove that $U$ is not dense if and only if there is a finite $S_0$ such that for every $i$ one has $\supp(\gota_i)\cap S_0\neq\emptyset$. \begin{itemize}
\item If $U$ is not dense, then there are $x\in\da$ and $\gotb\in\cali(D)$ such that
$$(x+\widehat\gotb)\cap\widehat\gota_i=\emptyset\;\forall\,i\in\N\,.$$
This implies $\widehat\gotb+\widehat\gota_i\neq\da$, which is equivalent to $\supp(\gotb)\cap\supp(\gota_i)\neq\emptyset$. Take $S_0=\supp(\gotb)$.
\item Assume $S_0$ exists. Then 
$$U\subseteq\bigcup_{\gotp\in S_0}U(\gotp)\,,\text{ where }U(\gotp):=\bigcup_{\gota_i\subseteq\gotp}\widehat\gota_i\,.$$
Clearly $U(\gotp)\subseteq\widehat\gotp$ and so the closure of $\cup_{\gotp\in S_0}U(\gotp)$ is contained in $\cup_{\gotp\in S_0}\widehat\gotp$ (since $S_0$ is finite). Hence $U$ cannot be dense. \end{itemize} \end{prf}

\begin{cor} \label{c:kpfree} Let $\calq$ be any infinite subset of $\calp(D)$ and $\gotp\mapsto k_\gotp$ any function $\calq\rightarrow\N$. The set
\begin{equation} \label{e:kpfree} \{x\in\da\mid v_\gotp(x)<k_\gotp\text{ for every prime }\gotp\in\calq\}\end{equation}
has empty interior. \end{cor}

\begin{prf} It is enough to observe that the set in \eqref{e:kpfree} is the complement of $U=\cup_{\gotp\in\calq}\gotp^{k_\gotp}$ and apply Lemma \ref{l:Udns}. (Note that $U=D$ if $k_\gotp=0$ for some $\gotp$.) \end{prf}

In particular, this applies when there is a fixed $k\ge2$ such that $k_\gotp=k$ for every $\gotp$; in this case, elements in the set in \eqref{e:kpfree} are usually called $k$-free.

\subsubsection{The complements of $U$} \label{sss:cmplU} We consider a sequence $(\gota_i)_{i\in\N}$ and $U_n,U$ as in \eqref{e:dfU}. Let 
$$X_n:=D-U_n\,,\,X:=D-U\,\text{ and }\,Y:=\da-U\,.$$
Note that the definition yields
\begin{equation} \label{e:vltY} Y=\big\{y\in\da\mid\forall\,i\;\exists\,\gotp\text{ such that }v_\gotp(y)<v_\gotp(\gota_i)\big\}. \end{equation}
Recalling that completions of ideals are both closed and open, one immediately sees that so are all $U_n$, and hence $\xa_n=\da-U_n$; moreover $U$ is open and $Y=\cap\xa_n$. It is also clear that $Y$ contains $\xa$ and it is natural to ask what is their difference.
 
\begin{rmk} \label{r:XvY} The inclusion $\xa\subseteq Y$ can be strict. For a simple example, take $D=\Z$ and let $\{\gota_i\}$ be  the set of all prime ideals: then $X=\{\pm1\}=\xa\neq 
Y=\za^*$.

Moreover, one can have $\rho(\xa)\subsetneq\rho(Y)$ - and hence $\zeta_{\rho(\xa)}(s)<\zeta_{\rho(Y)}(s)$ for every $s>1$. The family of ideals $\{\gota_0\}\cup\calj$ used in Lemma \ref{l:prmnc} provides an instance of this phenomenon: putting 
$$U=\bigcup_\calj\widehat\gotb\,,\;U_0=U\cup\widehat\gota_0\,,\;Y=\da-U\,,\;X=D-U\text{ and }Y_0=\da-U_0\,,$$
then Remark \ref{r:prmnc} shows $\gota_0\in\rho(U_0)-\rho(U)$, while inclusion \eqref{e:prmnc} yields $D\cap U=D\cap U_0$ and hence $X=Y\cap D=Y_0\cap D$. Thus $\xa$ is a subset of $Y_0$ 
and $\gota_0\in\rho(Y)-\rho(\xa)$.  \end{rmk}

\begin{lem} \label{l:YdaX} Assume $D$ has class number $1$. Then $Y=\xa\da^*$. \end{lem} 

\begin{prf} By \eqref{e:bln}, it is enough to prove $\rho(Y)=\rho(\xa)$.

Choose an increasing sequence $(S_n)_{n\in\N}$ of finite subsets of $\calp(D)$, so that $\cup_nS_n=\calp(D)$. Fix $y\in Y$ and, for every $n\in\N$ and $\gotp\in\calp(D)$, define
\begin{equation} \label{e:expny} e_n(\gotp):=\begin{cases} \begin{array}{ll} 0 & \text{ if }\gotp\notin S_n \\
v_\gotp(y) & \text{ if }\gotp\in S_n \text{ and }v_\gotp(y)\neq\infty\\
n & \text{ if }\gotp\in S_n \text{ and }v_\gotp(y)=\infty \end{array}\end{cases} \end{equation}
and consider the ideals
\begin{equation} \label{e:cny} \gotc_n:=\prod_\gotp\gotp^{e_n(\gotp)}\in\cali(D). \end{equation}
The hypothesis on the class number implies that for every $n$ the ideal $\gotc_n$ is principal: let $x_n\in D$ be a generator. One has $v_\gotp(x_n)=e_n(\gotp)\le v_\gotp(y)$ for every $n$ and $\gotp$ by construction and \eqref{e:vltY} then shows that each $x_n$ is in $X=Y\cap D$. Moreover $v_\gotp(y)=\lim e_n(\gotp)$ holds for every $\gotp$ and by Remark \ref{r:alxcmp} this is equivalent to $\lim\rho(x_n)=\rho(y)$. This proves $\rho(y)\in\rho(\xa)$ and hence $\rho(Y)=\rho(\xa)$. \end{prf}

The examples in Remark \ref{r:XvY} prove that the class number $1$ hypothesis in Lemma \ref{l:YdaX} is necessary and suggest that one cannot hope for much more than this without stronger conditions either on $D$ (see points 1 and 3 of Remark \ref{r:deN} for the case $D=\Z$) or on the sequence $(\gota_i)$. Theorem \ref{t:chbt} allows to prove the equality $\xa=Y$ in many cases. We give an example.

\begin{prop} \label{p:appirr} Assume that the ideals $\gota_i$ are pairwise coprime and satisfy $\Omega(\gota_i)>1$ for almost every $i$. Then $\xa=Y$. \end{prop}

The proof is in two steps: first we show $Y=\da^*\xa$ using the hypotheses on $(\gota_i)$ to approximate any element in $\rho(Y)$ by a sequence of principal ideals with generators in $X$; then we prove $\da^*\xa=\xa$  by means of Theorem \ref{t:chbt}.

\begin{prf} Define $S_n:=\cup_{i\le n}\supp(\gota_i)$. For simplicity of exposition we start making the additional assumption $\cup_nS_n=\calp(D)$.

Fix $y\in Y$ and define $e_n(\gotp)$ and $\gotc_n$ as in \eqref{e:expny} and \eqref{e:cny}. The reasoning in the proof of Lemma \ref{l:YdaX} shows that the sequence $\gotc_n$ converges to $\rho(y)$ in $\cals(D)$. For every $n$, choose a  prime ideal $\gotq_n$ such that $\gotc_n\gotq_n$ is principal and $\gotq_n\notin S_n\cup\{\gotq_0,...,\gotq_{n-1}\}$. From \eqref{e:vltY} we get
$$\rho(Y)=\big\{\sigma\in\cals(D)\mid\forall\,i\;\exists\,\gotp\text{ such that }v_\gotp(\sigma)<v_\gotp(\gota_i)\big\}.$$
By construction we have 
$$v_\gotp(\gotc_n\gotq_n)=\begin{cases} \begin{array}{ll} v_\gotp(\gotc_n)\le v_\gotp(y) & \text{ if }\gotp\in S_n\,; \\
v_\gotp(\gotq_n)\le1 & \text{ if }\gotp\notin S_n\,. \end{array}\end{cases}$$
The first inequality shows $\gotc_n\gotq_n\notin\rho(\widehat\gota_i)$ for $i\le n$. The second inequality implies $\gotc_n\gotq_n\notin\rho(\widehat\gota_i)$ for $i>n\gg0$, because $\gotc_n$ is not a multiple of $\gota_i$ and $\Omega(\gota_i)>1=\Omega(\gotq_n)$. Therefore $\gotc_n\gotq_n\in\rho(Y)$ for $n\gg0$. Since the ideals $\gotq_n$ are pairwise distinct, we get 
$$\lim_{n\rightarrow\infty}v_\gotp(\gotc_n\gotq_n)=\lim_{n\rightarrow\infty}v_\gotp(\gotc_n)=v_\gotp(y)\;\;\forall\,\gotp\in\calp(D)$$ 
and hence $\rho(y)=\lim_n\gotc_n\gotq_n$\,. 

Let $x_n\in D$ be a generator of $\gotc_n\gotq_n$, so that $\rho(x_n)=\gotc_n\gotq_n$\,. Then $x_n\in X$ for $n\gg0$ and the compactness of $\da$ ensures that $(x_n)$ has a limit $x$ (restricting, if needed, to a subsequence). Continuity yields $\rho(x)=\rho(y)$ and hence there is some $u\in\da^*$ such that $xu=y$.

For every $n$, let $i(n)$ be the smallest index such that $\gotq_n$ divides $\gota_{i(n)}$ and put 
$$T_n=S_n\cup\supp(\gota_{i(n)})\,.$$ 
By Theorem \ref{t:chbt}, one can find a sequence $(r_n)$ in $\Irr(D)$ such that 
$$r_n-u\in\prod_{\gotp\in T_n}\gotp^n\,.$$
Then $(r_n)$ converges to $u$ and $r_n\in\Irr(D)-T_n$. The latter condition ensures $x_nr_n\in X$ for $n\gg0$, as one can see by the same reasoning used for checking $\gotc_n\gotq_n\in\rho(Y)$. Therefore $y=\lim x_nr_n\in\xa$.

Finally, we observe that one can dispense with the assumption $\cup_nS_n=\calp(D)$ without any loss of generality. Indeed, having prime ideals outside $\cup_nS_n$ just provides more elbow 
room in the choice of $\gotq_n$ and $r_n$. \end{prf}

\subsection{The Davenport-Erd\H os theorem} The notation is the same as in \S\ref{sss:cmplU}: we fix a sequence of ideals $(\gota_i)$ and by taking their complements we obtain sets $X_n,\xa_n$, $X$, $Y$. The goal is to compare analytic density and measure, as defined in \eqref{e:d-anltid} and \eqref{e:romsr}, for $\widehat\calx_n:=\rho(\xa_n)$, $\widehat\calx:=\rho(\xa)$ and $\caly:=\rho(Y)$. We also fix an asymptotic density $\dsas$ on $D$ (as in \S\ref{sss:asnt}). \\

The sets $\xa_n$ and $Y$ are invariant under the action of $\da^*$ and therefore their measures in $\da$ coincide with those of $\widehat\calx_n$ and $\caly$ in $\cals(D)$. The inclusion-exclusion principle yields
\begin{equation} \label{e:inclescl} \mu(\xa_n)=\sum_{k\le n}(-1)^k\sum_{\substack{J\subseteq\{0,...,n\}\\ |J|=k}}\mu\left(\bigcap_{j\in J}\widehat\gota_j\right)\,. \end{equation}
Since (Dn$_\cali$7) and Lemma \ref{l:d=m0} can be applied to each summand in \eqref{e:inclescl}, the equality
\begin{equation} \label{e:dvnerdn} \dsas(X_n)=\mu(\xa_n)=\dsan(\widehat\calx_n)  \end{equation} 
holds for every $n$. It is natural to ask what happens when $n$ goes to infinity. The sets $\xa_n$ form a decreasing nested sequence with intersection $Y$: therefore $\mu(\xa_n)$ converges to $\mu(Y)$ and one has $\dsan(\widehat\calx_n)\ge\dsanp(\caly)$ for every $n$, yielding
\begin{equation} \label{e:m>dan+} \mu(Y)=\lim_{n\rightarrow\infty}\mu(\xa_n)\ge\dsanp(\caly)\,. \end{equation}

\begin{thm}[Davenport-Erd\H os] \label{t:de} The analytic density of $Y$ exists for every sequence $(\gota_i)$ and it satisfies $\dsan(\caly)=\mu(Y)=\rho_*\mu(\caly)$.  \end{thm}

Davenport and Erd\H os gave two proofs of this result, in \cite{de1} and \cite{de2}, for the case $D=\Z$. Our version extends it to any global Dedekind domain and every sequence of ideals, without any of the additional hypotheses on either $D$ or $(\gota_i)$ we needed in \S\ref{sss:cmplU}. We shall prove the theorem in \S\ref{sss:prfde}, following quite closely the reasoning of \cite{de1}. 

\begin{rmks} \label{r:deN} \begin{itemize} \item[]
\item[{\bf 1.}] In \eqref{e:dvnerdn} we also inserted an asymptotic density. At this stage we cannot say anything about $\dsas^\pm(X)$ in full generality. However, in the case $D=\Z$ the chain of inequalities
\begin{equation} \label{e:as-an} \dsasm(W)\le\dlib_{\log}^-(W)\le\dlib_{\log}^+(W)\le\dsasp(W)\, \end{equation}
holds for every $W\subseteq\Z$; moreover, the logarithmic density $\dlib_{\log}(W)$ exists if and only if so does $\dsan(W)$ and when this happens they are equal (see e.g. \cite[III.1, 
Theorems 1.2 and 1.3]{tnm}). Because of class number $1$, the functions $\zeta_X$ and $\zeta_Y$ are the same and Theorem \ref{t:de} leads to
$$\mu(\xa)\le\mu(Y)=\dsan(Y)=\dsan(X)\stackrel{\text{by \eqref{e:as-an}}}{\le}\dsasp(X)\le\mu(\xa)\,,$$
proving 
$$\dsasp(X)=\mu(\xa)=\mu(Y).$$ 
This is the best  possible result, since Besicovitch in \cite{bsc} gave an example of a set of multiples for which the asymptotic density does not exist.
\item[{\bf 2.}] Put $\calx=\rho(X)$. If $D$ has class number $1$, then $\caly\cap\cali(D)=\calx\cap\cali(D)$ implies $\dsan(\calx)=\dsan(\caly)$. When the class group of $D$ is not trivial, one has to look at $\widehat\calx$ (since $\calx$ contains only principal ideals):  in this case we do not know if $\widehat\calx$ has analytic density and, if yes, whether it is the same as the one of $\caly$.
\item[{\bf 3.}] The equality $\mu(\xa)=\mu(\da^*\xa)$ holds for $D=\Z$, thanks to comparison with $\dsasp$\,, as seen above. It is not clear whether it is always true with a general $D$, or even just in the case of class number $1$. The stronger equality $\mu(Y)=\mu(\xa)$ holds in all the examples of Remark \ref{r:XvY}, because in all of them we have $\mu(Y)=0$. It would be interesting to have an example with $\mu(Y)\neq\mu(\xa)$.
\item[{\bf 4.}] In the case $D=\Z$, it is proved in \cite[Theorem 3]{eht} that $\dsas(X)$ esists when $|\supp(\gota_i+\gota_j)|$ is bounded. This result might be extendable to our setting: Theorem \ref{t:asdmltp} is a first step in this direction.
\item[{\bf 5.}] Classically, the Davenport-Erd\H os theorem is expressed looking at $U$ (the set of multiples) rather than $Y$ or $X$ and it is usually formulated as the equality $\dsl(U)=\mu(U)$. We considered the complements of $U$ because in our setting it is interesting to ask the finer question of determining the closure in $\da$ of a subset of $D$, rather than just its density. However, our final result is about $\caly$, a subset of $\cals(D)$. One could perhaps try and adapt the reasoning in \cite{de2} in order to compute $\dsl(X)$ (the logarithmic density on $D$, as defined in \cite{dl3}). A proof along the lines of \cite{de2} appeared in \cite{huck}, for a logarithmic density defined on $\cali(\calo_K)$ (the ring of integers of a number field). \end{itemize} \end{rmks}
 
\subsubsection{Some formulae} Before proving Theorem \ref{t:de}, we collect some analogues of well-known formulae from analytic number theory. The proofs are standard exercises and are left to the reader. First of all, from \eqref{e:ztX} one obtains
\begin{equation} \label{e:dztX} \frac{d}{ds}\zeta_\calz(s)=\sum_{\gotn\in\calz}\frac{d}{ds}\frac{1}{\|\gotn\|^s}=-\sum_{\gotn\in\calz}\frac{\log\|\gotn\|}{\|\gotn\|^s} \end{equation}
for every $\calz\subseteq\cali(D)$. Also, we define the von Mangoldt function on $\cali(D)$ by
$$\Lambda_D(\gotn)=\begin{cases} \begin{array}{ll} \log\|\gotp\| & \text{ if }\gotn=\gotp^r\text{ for some }r\in\N\,;\\ 0 & \text{ if }\gotn\text{ is not a prime power}. \end{array}\end{cases}$$
As in the classical case, unique factorization (of ideals) immediately implies 
\begin{equation} \label{e:vm} \log\|\gotn\|=\sum_{\gotd|\gotn}\Lambda_D(\gotd)  \end{equation}
and, with slightly more effort,
\begin{equation} \label{e:dlg} \dlog\zeta_D(s)=-\sum_\gotn\frac{\Lambda_D(\gotn)}{\|\gotn\|^s} \end{equation}
where $\dlog$ denotes the logarithmic derivative (that is, the operator sending a meromorphic function $f$ to $\frac{f'}{f}$).

\subsubsection{Proof of the Davenport-Erd\H os theorem} \label{sss:prfde} We start with the analogue of \cite[Lemma 1]{de1}.

\begin{lem} \label{l:de} For every $n$, the ratio $\delta(\widehat\calx_n,s)$ of \eqref{e:dltXs} extends to an increasing continuous function on $[1,\infty)$. \end{lem}

\begin{prf} For a general $\calz\subseteq\cals(D)$, it is clear from \eqref{e:dltXs} and \eqref{e:d-anltid} that $\delta(\calz,s)$ is continuous on $(1,\infty)$ and can be continuously extended to $1$ if and only if $\dsan(\calz)$ exists. We have already seen that this is the case for our sets $\widehat\calx_n$\,.

In order to prove monotonicity, put $\calu_n:=\rho(U_n)$ and observe that one has 
$$\delta(\widehat\calx_n,s)=1-\delta(\calu_n,s)\,,$$
because $\widehat\calx_n$ and $\calu_n$ are both compact open and $\cals(D)$ is their disjoint union (both statements follow from Remark \ref{r:daqt}, using the fact that $U_n$ and $\xa_n$ are both compact open and $\da^*$-invariant). We shall show that the derivative $\frac{d}{ds}\delta(\calu_n,s)$ is non-positive for every $s>1$. 

Since $\delta(\calu_n,s)>0$ is obviously true, it is enough to show
$$0\ge\dlog(\delta(\calu_n,s))=\dlog\zeta_{\calu_n}(s)-\dlog\zeta_D(s)\,,$$ 
which is clearly equivalent to
\begin{equation} \label{e:dver} \zeta_{\calu_n}(s)\cdot\dlog\zeta_D(s)\ge\frac{d}{ds}\zeta_{\calu_n}(s) \end{equation}
By \eqref{e:ztX}, \eqref{e:dztX} and \eqref{e:dlg}, we can rewrite \eqref{e:dver} as
$$\sum_{\gotn\in\cali(D)}\frac{\buno_n(\gotn)}{\|\gotn\|^s}\cdot\left(-\sum_{\gotn\in\cali(D)} 
\frac{\Lambda_D(\gotn)}{\|\gotn\|^s}\right)\ge-\sum_{\gotn\in\cali(D)}\frac{\buno_n(\gotn)\log\|\gotn\|}{\|\gotn\|^s} $$
where $\buno_n$ is the characteristic function of $\calu_n$. The last inequality follows if we can prove 
\begin{equation} \label{e:vm>} \buno_n(\gotn)\log(\|\gotn\|)\ge\sum_{\gotd|\gotn}\buno_n(\gotd)\Lambda_D(\gotn\gotd^{-1})\,. \end{equation}
If $\buno_n(\gotn)=1$, then \eqref{e:vm>} is immediate from \eqref{e:vm}. In the other case, \eqref{e:idmlt} shows that an ideal is in $\calu_n$ if and only if it is a multiple of some 
$\gota_i$, with $i\le n$: hence $\buno_n(\gotn)=0$ implies $\buno_n(\gotd)=0$ for every $\gotd$ dividing $\gotn$ and \eqref{e:vm>} reduces to $0=0$. \end{prf}

\begin{proof}[Proof of Theorem \ref{t:de}] By equality \eqref{e:dvnerdn} and Lemma \ref{l:de} it follows
\begin{equation} \label{e:dlt>m}  \delta(\widehat\calx_n,s)\ge\delta(\widehat\calx_n,1)=\dsan(\widehat\calx_n)=\mu(\xa_n)\ge\mu(Y)\,.  \end{equation}
Moreover we have
$$\lim_{n\rightarrow\infty}\zeta_{\widehat\calx_n}(s)=\zeta_\caly(s)\;\;\;\;\forall\,s\in(1,\infty)$$
because $\cap_n\widehat\calx_n=\caly$. Therefore as $n$ varies the functions $\delta(\widehat\calx_n,s)$ converge pointwise to $\delta(\caly,s)$ on $(1,\infty)$ and \eqref{e:dlt>m} yields
$$\delta(\caly,s)\ge\mu(Y)$$
for every $s\in(1,\infty)$, which implies $\dsanm(\caly)\ge\mu(Y)$. Now use \eqref{e:m>dan+}. \end{proof}

\section{Eulerian sets} \label{s:elr}

In this section $D$ is any Dedekind domain satisfying Assumption \ref{a:cntb} and $d$ is a density on $D^n$ such that (Dn1)-(Dn7) hold.

\subsection{Eulerian sets}

\begin{dfn} \label{d:elr} Let $X$ be a subset of $\da^n$. \begin{enumerate}
\item For any prime ideal $\gotp$ of $D$, we write $X(\gotp)$ to denote the closure of $\hpi_{\gotp^\infty}(X)$ in $\dap^n$.
\item We say that $X\subseteq D^n$ is {\em Eulerian} if 
$$\xa=\prod_\gotp X(\gotp)$$
in the identification $\da=\prod\dap$ induced by \eqref{e:crt}.
\item We say that $X$ is {\em openly Eulerian} if it is Eulerian and each $X(\gotp)$ is open. \end{enumerate} \end{dfn}

\begin{eg} Ideals of $D$ provide obvious examples of openly Eulerian sets. More generally, given $X\subseteq\da^n$ and $\gota\in\cali(D)$ the set $X_\gota$ is openly Eulerian if and only if $\hpi_\gota(X)$ can be written as a product in the decomposition
$$(\da/\widehat\gota)^n\simeq\prod_{\gotp|\gota}(\da/\widehat\gotp^{v_\gotp(\gota)})^n$$
induced by the Chinese Remainder Theorem. \end{eg}

Let $\mu_\gotp$ denote the Haar measure on $\dap^n$ (normalized so to have $\mu_\gotp(\dap^n)=1$). Corollary \ref{c:egprd} yields
$$\mu\left(\prod_\gotp X(\gotp)\right)=\prod_\gotp\mu_\gotp(X(\gotp))\,.$$

\begin{lem} \label{l:d<prd} The inequality
\begin{equation} \label{e:banale2} d^+(X)\le\prod_\gotp\mu_\gotp(X(\gotp))\, \end{equation} 
holds for every $X\subseteq D^n$ and every density satisfying {\rm (Dn1)-(Dn7)}. \end{lem}

\begin{prf} For any $X$ one has $\prod_\gotp X(\gotp)=\bigcap_\gotp X_{\gotp^\infty}\,,$ because, by \eqref{e:Xpinf}, 
$$X_{\gotp^\infty}=X(\gotp)\times\prod_{\gotq\neq\gotp}\da_\gotq^n\;.$$
Thus the inclusion $\xa\subseteq\prod X(\gotp)$ is always true and \eqref{e:banale2} follows Corollary \ref{c:m>d}. \end{prf}

\begin{prop} \label{p:d>mc} Assume $\xa$ is contained in a set $C=\prod_\gotp C_\gotp$, where every $C_\gotp\subseteq\dap^n$ is open and $\prod\mu_\gotp(C_\gotp)>0$. Then 
$d^+(X)\ge\mu(C)$ implies $\xa=C$ and $X(\gotp)=C_\gotp$ for every $\gotp$. \end{prop}

\begin{prf} Assume $\xa\neq C$. Then there is an open set $U$ such that $\xa\subseteq C-U$ and $U\cap C\neq\emptyset$. Since \eqref{e:base} is a base for the topology, there is no loss of 
generality in assuming $U=\hpi_\gota^{-1}(x)$ for some $\gota\in\cali(D)$ and some $x=\pi_\gota(\tilde x)\in(D/\gota)^n$, that is, $U=\prod_\gotp(\tilde x+\gota \dap^n)$. One has 
$\gota\dap^n=\dap^n$ unless $\gotp|\gota$ and thus
$$U\cap C=\prod_{\gotp|\gota}B_\gotp\times\prod_{\gotp\nmid\gota}C_\gotp\,,$$
where $B_\gotp=C_\gotp\cap(\tilde x+\gota \dap^n)$ is an open subset of $\dap$. Hence
$$\mu(U\cap C)=\prod_{\gotp|\gota}\mu_\gotp(B_\gotp)\cdot\prod_{\gotp\nmid\gota}\mu_\gotp(C_\gotp)>0\,,$$
because $\mu_\gotp(B_\gotp)>0$ for each of the finitely many $\gotp$ dividing $\gota$. But then one gets a contradiction from
$$\mu(\xa)\stackrel{\text{by \eqref{e:m>d}}}{\ge} d^+(X)\ge\mu(C)>\mu(C-U)\ge\mu(\xa)\,.$$
Finally, $\xa=C$ implies $X(\gotp)=\hpi_{\gotp^\infty}(\xa)=\hpi_{\gotp^\infty}(C)=C_\gotp$. \end{prf}

\begin{cor} \label{c:d>mc} Assume every $X(\gotp)$ is open and $\prod\mu_\gotp(X(\gotp))>0$. Then $X$ is openly Eulerian if \eqref{e:banale2} is an equality. \end{cor}

\begin{prf} Apply Proposition \ref{p:d>mc} with $C_\gotp=X(\gotp)$. \end{prf}

\subsubsection{Eulerianity and polynomial maps} Polynomial maps are well-behaved with regard to Eulerianity.

\begin{prop} \label{p:plyneulr} For any $f\in D[x_1,...,x_n]$ and $X\subseteq D^n$, one has: \begin{itemize}
\item[{\bf (a)}] if $X$ is Eulerian, then so is $f(X)$;
\item[{\bf (b)}] if $f(X)$ is Eulerian and $\xa$ is the closure of $f^{-1}\big(f(X)\big)$, then $X$ is Eulerian.\end{itemize} \end{prop}

\begin{prf} Polynomial maps commute with ring homomorphisms: hence, by \eqref{e:crt}, we have
\begin{equation} \label{e:plynmprd} f\left(\prod_\gotp A_\gotp\right)=\prod_\gotp f(A_\gotp)\;\;\text{ and }\;\;f^{-1}\left(\prod_\gotp B_\gotp\right)=\prod_\gotp f^{-1}(B_\gotp) \end{equation}
for all $A_\gotp\subseteq\dap^n$, $B_\gotp\subseteq\dap$\,. Besides, the map induced by $f$ is continuous on $\da^n$ and on each $\dap^n$. \begin{itemize}
\item[{\bf (a)}] Since $X$ is dense in $X(\gotp)$ and $X(\gotp)$ is compact, continuity implies that the closure of $f(X)$ in $\dap$ is 
exactly $f(X(\gotp))$. Replacing $\dap$ with $\da$, the same reasoning shows that $f(\xa)$ is the closure  of $f(X)$ in $\da$. Now just apply \eqref{e:plynmprd} to $\xa=\prod X(\gotp)$.
\item[{\bf (b)}] Put $Y:=f(X)$. Then we have $\widehat Y=\prod_\gotp Y(\gotp)$ and $\xa=\widehat{f^{-1}(Y)}=f^{-1}(\widehat Y)$ by hypothesis, so \eqref{e:plynmprd} yields
$$\xa=\prod_\gotp f^{-1}(Y(\gotp))$$
and hence $X(\gotp)=\hpi_{\gotp^\infty}(\xa)=f^{-1}(Y(\gotp))$. \end{itemize} \end{prf}

\begin{rmk} In general, the inverse image of an Eulerian set by a polynomial map is not Eulerian: for a trivial example, consider $f(x)=x(x-1)$ and note that the set $\{0\}$ is Eulerian, but $f^{-1}(\{0\})$ is not (since it contains two points in $D$ and infinitely many in $\da$, because the latter is a product of the rings $\dap$). Thus the condition on $\xa$ in Proposition \ref{p:plyneulr}.(b) is somehow necessary. \end{rmk}

\subsubsection{Eulerianity and strong approximation} \label{sss:strapprx} We say that an affine $D$-scheme $Z$ satisfies the strong approximation theorem with respect to  $D$ if $Z(D)$ is dense in $Z(\da)$ (with respect to the obvious topology as a subset of $\da^n$). For example, the strong approximation theorem holds for the affine space $\A^n$ and for the special linear group $\SL_n$ (see \cite[VII, \S2, n.4]{b-ac}), but not for $\GL_n$.

\begin{rmk} \label{r:strng} Usually strong approximation is expressed in terms of adeles. More precisely, let $F$ be the fraction field of $D$ and consider the ring of $D$-adeles $\hat 
F:=F\otimes_D\da$ (if $D$ is the ring of integers of a number field, then $\hat F$ is the usual ring of finite adeles). The more common statement is that the strong approximation theorem is true for $Z$ if $Z(F)$ is dense in $Z(\hat F)$. The sets $\hat F$ and $Z(\hat F)$ are given the restricted product topology: it follows that $Z(\da)$ is compact open in $Z(\hat F)$. In particular, any open subset $U\subseteq Z(\da)$ is also open in $Z(\hat F)$ and therefore $U\cap Z(F)\neq\emptyset$ if strong approximation holds. One can conclude that the usual version of strong approximation implies the one we have given above by observing $Z(D)=Z(\da)\cap Z(F)$.  \end{rmk}

By functoriality, the isomorphism \eqref{e:crt} yields $Z(\da)=\prod Z(\dap)$. If $X=Z(D)$ for some affine scheme $Z$, then $X(\gotp)=Z(\dap)$ for every $\gotp$ and therefore $X$ is Eulerian if and only if $Z$ satisfies the strong approximation theorem. However, the notion of Eulerian allows for more generality. For example, the set of square-free integers is Eulerian (as we shall prove in Corollary \ref{c:kfrelr}), but as far as we know it cannot be obtained as the set of $\Z$-points of a scheme.

\begin{prop} \label{p:str} Let $G$ be an algebraic subgroup of $\GL_n$ satisfying the strong approximation theorem. Let $X$ be a subset of $D^n$, stable under the action of $G(D)$. If $X$ is contained in a $G(\da)$-orbit, then it is Eulerian. \end{prop}

\begin{prf} The group 
$$\GL_n(\da)=\varprojlim\GL_n(D/\gotn)$$ 
has the inverse limit topology and therefore it is Hausdorff and compact. Moreover, the action of $\GL_n(\da)$ on $\da^n$ is continuous: it follows that for any subgroup $H$ of $\GL_n(\da)$, the closure of an $H$-orbit is the $\widehat H$-orbit of the same element. 

For $H=G(D)$ strong approximation yields $\widehat H=G(\da)$. The $G(D)$-orbit of $v\in X$ is still in $X$ and its closure is a $G(\da)$-orbit: it follows that $\xa$ is a $G(\da)$-orbit. Finally, the equality $G(\da)=\prod G(\dap)$ shows that a $G(\da)$-orbit is a product of $G(\dap)$-orbits (hence Eulerian).  \end{prf}
 
\begin{rmk} In particular, Proposition \ref{p:str} applies when $X$ is a $G(D)$-orbit. \end{rmk}

Let $R$ be a ring. The {\em greatest common divisor} of $a_1,\dots,a_n\in R$ is the ideal generated by these elements. In particular, we say that the $n$-tuple $(a_1,...,a_n)$ is coprime if the greatest common divisor of its elements is $R$.

\begin{cor} \label{c:cprpr} For $\gota\in\cali(D)$ and $n\ge2$, let $X\subset D^n$ be the set of all $n$-tuples with greatest common divisor $\gota$. Then $X$ is Eulerian. \end{cor}

\begin{prf} Let $Y\subseteq\da^n$ be the set of all $n$-tuples with greatest common divisor $\widehat\gota$. Then $X=Y\cap D^n$. Lemma \ref{l:mxcmdv} below shows that $Y$ consists of one 
$SL_n(\da)$-orbit; in particular this implies that $Y$ is $\SL_n(D)$-stable, as obviously is $D^n$, and thus so is $X$. As already mentioned strong approximation holds for $\SL_n$ and so all conditions in Proposition \ref{p:str} are satisfied. \end{prf}

\begin{lem} \label{l:mxcmdv} For any $n\ge1$, let $Y$ be the set of all $n$-tuples with greatest common divisor $\widehat\gota$ in $\da^n$. Then $Y$ consists of a single $\SL_n(\da)$-orbit. \end{lem}

\begin{prf} Elements of $D^n$ and $\da^n$ will be thought of as column vectors, so that $\GL_n$ acts on the left.

By Lemma \ref{l:clpr}, the ideal $\widehat\gota$ is principal: let $a\in\da$ be a generator. We claim that $Y$ is the $\SL_n(\da)$-orbit of the vector $v_a$ with entries $a,0,\dots,0$.

For any $g\in\SL_n(\da)$, entries of the first row form a coprime $n$-tuple: hence $\SL_n(\da)\cdot v_a\subseteq Y$. For the opposite inclusion, we proceed by induction on $n$. The claim is trivial for $n=1$. For $n\ge2$, assume $a_1,\dots,a_n$ are the entries of $v\in Y$ and consider the vector $w\in\da^{n-1}$ whose entries $a_1,\dots,a_{n-1}$ have greatest common divisor $b\da$. By the induction hypothesis there is $g\in\SL_{n-1}(\da)$ such that $g\cdot v_b=w$. Besides we can write $b=ac_1$ and $a_n=ac_2$ where $c_1,c_2$ are coprime. Thus there is $h\in\SL_2(\da)$ having $c_1$ and $c_2$ in the first column. By appropriate embeddings of $\SL_{n-1}$ and $\SL_2$ in $\SL_n$, we can think of $g$ and $h$ as elements of $\SL_n(\da)$ so to have $hg^{-1}\cdot v_a=v$. \end{prf}

With no additional effort, one can give a more precise description of the sets in Corollary \ref{c:cprpr} and Lemma \ref{l:mxcmdv}. Let $Y_1$ be the set of coprime $n$-tuples in 
$\da^n$. Then for any $a\in\da$ the set of $n$-tuples with greatest common divisor $a\da$ is $aY_1$ (since - in the notation of the proof above - one has $v_a=av_1$). Besides 
$v$ is in $Y_1$ if and only if for every prime ideal $\gotp$ there is at least one entry of $v$ which is not in $\gotp\da$ - that is, 
$$Y_1=\prod_{\gotp\in\calp(D)}\big(\dap-\gotp\dap\big)^n\,.$$
Thus if $X\subset D^n$ is the set of all $n$-tuples with greatest common divisor $\gota$ and $a\in\da$ is a generator of $\widehat\gota$, we obtain
$$\xa=aY_1=a\prod_{\gotp\in\calp(D)}\big(\dap-\gotp\dap\big)^n\,,$$
which has measure
\begin{equation} \label{e:prbbgcd} \mu(\xa)=\frac{1}{\|\gota\|^n}\prod_{\gotp\in\calp(D)}\left(1-\frac{1}{\|\gotp\|^n}\right)=\frac{1}{\|\gota\|^n\cdot\zeta_D(n)}\;, \end{equation}
with $\zeta_D$ as in \eqref{e:Ddkzt}.

\subsubsection{Complements of ideals and sets of Eulerian type} \label{sss:appirr} Given a partition $\calp(D)=\bigsqcup_{i\in\N}T_i$, define
\begin{equation} \label{e:diTi} \da_i:=\prod_{\gotp\in T_i}\dap\;\;\text{ and }\;\;\hpi_i:=\prod_{\gotp\in T_i}\hpi_{\gotp^\infty}\,, \end{equation}
so to have $\hpi_i\colon\da\twoheadrightarrow\da_i$ and $\da\simeq\prod_i\da_i$. Also, for $X\subseteq D^n$ let $X(T_i)$ be the closure of $\hpi_i(X)$ in $\da_i^n$. 

\begin{dfn} \label{d:elrtyp} A set $X\subseteq D^n$ is {\em of Eulerian type} if there exists a partition of $\calp(D)$ into finite subsets $T_i$ such that
$$\xa=\prod_{i\in\N}X(T_i)\,.$$
We say that $X$ is {\em of openly Eulerian type} if moreover each $X(T_i)$ is open in $\da_i^n$. \end{dfn}

\begin{prop} \label{p:cprelrtyp} Assume $D$ is a global Dedekind domain. Let $X=D-\cup_i\gota_i$ where $(\gota_i)_{i\in\N}$ is a sequence of pairwise coprime ideals of $D$ satisfying 
$\Omega(\gota_i)>1$ for almost every $i$. Then: \begin{enumerate}
\item $X$ is of openly Eulerian type;
\item $X$ is Eulerian if and only if each $\gota_i$ is a prime power. \end{enumerate} \end{prop}

\begin{prf} As in \S\ref{sss:cmplU}, we put $Y=\da-\cup_i\widehat\gota_i$\,. Also, let $\calp_0$ denote the complement of $\cup_i\supp(\gota_i)$ in $\calp(D)$. We consider the partition
$$\calp(D)=\bigsqcup_{i\in\N}\supp(\gota_i)\,\sqcup\bigsqcup_{\gotp\in\calp_0}\{\gotp\}$$ 
and we abbreviate $X(\supp(\gota_i))$ into $X(\gota_i)$. Statement (1) follows by Proposition \ref{p:appirr} and the chain of inclusions
\begin{equation} \label{e:elrtyp} \xa\subseteq\prod_{i\in\N}X(\gota_i)\times\prod_{\gotp\in\calp_0}X(\gotp)\subseteq\prod_{i\in\N}\big(D_i-\gota_iD_i\big)\times\prod_{\gotp\in\calp_0}\dap=Y \end{equation}
where $D_i=\prod_\gotp\dap$ with the product taken over $\supp(\gota_i)$. The first inclusion is straightforward by a reasoning already used in the proof of Lemma \ref{l:d<prd}. The characterization of $Y$ given in \eqref{e:vltY} implies the second inclusion and the final equality.

In statement (2) the ``if'' part is an obvious special case of statement (1). As for the ``only if'' part, one can understand why it holds by observing that if $\gotp$ and $\gotq$ are distinct primes and $\da_{\gotp\gotq}=\dap\times\da_\gotq$ then
$$\da_{\gotp\gotq}-\gotp\gotq\da_{\gotp\gotq}\neq(\dap-\gotp \dap)\times(\da_\gotq-\gotq\da_\gotq)\,.$$
\end{prf}

\begin{cor} \label{c:kfrelr} The set of $k$-free elements of a global Dedekind domain is openly Eulerian for any $k\ge2$. \end{cor}

\noindent (The definition of $k$-free was given just after Corollary \ref{c:kpfree}.)

\begin{rmk} One could get a different proof of Proposition \ref{p:cprelrtyp} by computing the (asymptotic) density of $X$ and then using an ``Eulerian type'' version of  Proposition \ref{p:d>mc}. Our proof has the advantage of illustrating a different technique, which offers a complete description of $\xa$ just by a qualitative argument. \end{rmk}

\begin{thm} \label{t:asdmltp} Let $D$ and $X$ be as in Proposition \ref{p:cprelrtyp}. Then the asymptotic density of $X$ exists and it satisfies 
$$\dsas(X)=\mu(\xa)=\prod_{i\in\N}\left(1-\frac{1}{\|\gota_i\|}\right).$$ 
\end{thm}

\noindent Here $\dsas=\dsbas$, where $B$ is as in \S\ref{sss:asnt}. The choice of $B$ has no influence on the final result, but it is useful to fix it for the computations in the proof.

\begin{proof} We adapt a strategy used in \cite[\S5]{ks1} to compute the density of square-free integers. (Alternatively, one could use Theorem \ref{t:mt}, in the form suggested by 
Proposition \ref{p:ps} below.)

For every $i$ let $t_i\in\da$ be a fixed generator of the closed (hence principal) ideal $\widehat\gota_i$ and put
$$\calk:=\left\{\prod_{i\in \N}t_i^{n_i}\mid n_i\in \N,\text{ }n_i=0\text{ almost everywhere}\right\}\,.$$
Then $\calk$ is countable, contains $1$ and (as we will check later) it enjoys the properties: \begin{itemize}
\item[(K1)] $k\xa\cap h\xa=\emptyset$ for any two distinct $h,k\in\calk$;
\item[(K2)] $D-\{0\}\subseteq\bigsqcup_{k\in \calk}k\xa$. \end{itemize}
Let $\calk_1:=\calk-\{1\}$. By (K1) we can write 
$$\bigcup_{k\in\calk}k\xa=\xa\sqcup W\;\;\text{ with }\;\;W:=\bigsqcup_{k\in\calk_1}k\xa\,,$$
which yields
\begin{equation} \label{e:kbtsgt} 1\ge\mu(\xa)+\mu(W)=\sum_{k\in\calk}\mu(k\xa)=\sum_{k\in \calk}\frac{\mu(\xa)}{\|k\|}\,. \end{equation}
Since $D-\{0\}$ has density $1$, property (K2) and formula \eqref{e:d+-} imply
\begin{equation} \label{e:k2} 1=\dsasm(X)+\dsasp(W)\,. \end{equation}

We want to apply Fatou's lemma on the measure space $\calk_1$, with the counting measure $\nu$. By definition, one has 
$$\int_{\calk_1}fd\nu=\sum_{k\in\calk_1}f(k)$$
for any integrable function $f$. For $k\in\calk_1$, let $X_k:=k\xa\cap D$ and $\gotk:=k\da\cap D$. Fix $B$ as in \S\ref{sss:asnt} and note that \eqref{e:dnsvlm} implies the existence of  positive constants $c_1,c_2\in\R$ such that
\begin{equation} \label{e:c1c2} c_1\prod_{v\in S}|r|_v^n\le|D^n\cap rB|\le c_2\prod_{v\in S}|r|_v^n \end{equation}
for every $r\in F^*$ which is sufficiently big in the sense of \eqref{e:rsffcbg}. Then \eqref{e:c1c2} and \eqref{e:k&rB} imply
$$|X_k\cap rB|\le|\gotk\cap rB|\le\frac{c_2+1}{\|k\|}\prod_{v\in S}|r|_v$$
if $r$ is sufficiently large. For $r\in F^*$, consider the function $f_r\colon\calk_1\rightarrow[0,1]$ defined by
$$f_r(k):=\frac{|X_k\cap rB|}{|D\cap rB|}\,.$$
By the above (and using \eqref{e:c1c2} to bound $|D\cap rB|$ from below) we have
$$f_r(k)\le\frac{c_2+1}{c_1\|k\|}=:g(k)$$
and $g$ is integrable on $\calk_1$ by \eqref{e:kbtsgt}. Therefore we can use Fatou to obtain
\begin{equation} \label{e:fatou}  \sum_{k\in\calk_1}\limsup_{r\rightarrow\infty}f_r(k)=\int_{\calk_1}(\limsup_{r\rightarrow\infty}f_r)d\nu\ge\limsup_{r\rightarrow\infty}\int_{\calk_1}f_rd\nu=\limsup_{r\rightarrow\infty}\sum_{k\in\calk_1}f_r(k)\,.  \end{equation}

The equalities
$$\limsup_{r\in F^*}f_r(k)=\dsbasp(X_k)=\dsasp(X_k)$$ 
and
$$\limsup_{r\in F^*}\sum_{k\in\calk_1}f_r(k)=\dsbasp(W)=\dsasp(W)$$
hold by definition of asymptotic density. Moreover, we have $\mu(k\xa)\ge\dsasp(k\xa)=\dsasp(X_k)$ for every $k$, by Corollary \ref{c:m>d}. Thus \eqref{e:fatou} yields
$$\mu(W)=\sum_{k\in \calk_1}\mu(k\xa)\ge\sum_{k\in\calk_1}\dsasp(X_k)\ge\dsasp(W)$$
and hence
$$\dsasm(X)\stackrel{\text{by \eqref{e:k2}}}{=}1-\dsasp(W)\ge1-\mu(W)\stackrel{\text{by \eqref{e:kbtsgt}}}{\ge}\mu(\xa)\stackrel{\text{by \eqref{e:m>d}}}{\ge}\dsasp(X)\,,$$
showing that $X$ has density $\mu(\xa)$. The evaluation of this measure as a product is straightforward from the fact that all inclusions in \eqref{e:elrtyp} are equalities.

For checking (K1), take $h=\prod t_i^{m_i}$ and $k=\prod t_i^{n_i}$, with $h\neq k$. Then there must be an index $i_0$ such that $n_{i_0}\neq m_{i_0}$ and we can assume $n_{i_0}<m_{i_0}$ without loss of generality. If (K1) were false, one would have $ky_1=hy_2$ for some $y_1,y_2\in\xa$. Putting $e_\gotp:=v_\gotp(t_{i_0})=v_\gotp(\gota_{i_0})$ for every $\gotp\in\supp(\gota_{i_0})$, we get
$$v_\gotp(y_1)+n_{i_0}e_\gotp=v_\gotp(ky_1)=v_\gotp(hy_2)=v_\gotp(y_2)+m_{i_0}e_\gotp\,,$$
which yields
$$v_\gotp(y_1)=v_\gotp(y_2)+(m_{i_0}-n_{i_0})e_\gotp\ge e_\gotp\,.$$
But, as $y_1\notin\widehat\gota_{i_0}$, there must exist a prime $\gotp$ such that $v_\gotp(\gota_{i_0})>v_\gotp(y_1)$, contradiction.

As for (K2), given $z\in D-\{0\}$ we consider the ideal 
$$z\da=\prod_{\gotp\in \calp(D)}\widehat\gotp^{v_\gotp(z)}\,.$$
This can be rewritten as $z\da=xk\da$, where $x\in \xa$ and $k\in \calk$, by rearranging the occurring powers of the involved primes. Therefore $z=xku$, with $x,k$ as before and $u\in\da^*$. As $xu\in\xa$, we conclude that (K2) holds. \end{proof}

\subsection{The Poonen-Stoll condition} 

The following statement can be seen as a reformulation (in a more general setting) of \cite[Lemma 1]{ps1}, \cite[Lemma 3.1]{bbl}, \cite[Theorem 2.1]{mch} and \cite[Proposition 3.4]{cs}.

\begin{prop} \label{p:ps} Let $(S_j)_{j\in\N}$ be an increasing sequence of finite subsets of $\calp(D)$, covering it. For every $\gotp\in\calp(D)$ consider a subset $U_\gotp\subseteq\dap^n$ and put 
$X:=D^n\cap\prod U_\gotp$. Let $d=(d^+,d^-)$ be any density as in \S\ref{sss:dns}. Assume that \begin{itemize}
\item[(PS1)] $\mu_\gotp(\partial U_\gotp)=0$ holds for every $\gotp$;
\item[(PS2)] the Poonen-Stoll condition
\begin{equation} \label{e:ps} \lim_{j\rightarrow\infty}d^+\left(\bigcup_{\gotp\notin S_j}\big(D^n-U_\gotp\big)\right)=0 \end{equation}
holds. \end{itemize}
Then the density $d(X)$ exists and it is equal to $\prod\mu_\gotp(U_\gotp)$. \end{prop}

Note that the hypothesis (PS1) implies that each $U_\gotp$ is measurable.

\begin{prf} The idea is to apply Theorem \ref{t:mt}, taking $\calt=\{\sigma_j\}_{j\in\N}$ with $\sigma_j:=\prod_{\gotp\in S_j}\gotp^\infty$.

Put $U:=\prod U_\gotp$. Then $U(\gotp)$ is the closure of $U_\gotp$ in $\dap^n$. Hypothesis (PS1) yields $\mu_\gotp(U_\gotp)=\mu_\gotp(U(\gotp))$. The equality 
$$U_{\sigma_j}=\prod_{\gotp\in S_j}U(\gotp)\times\prod_{\gotp\notin S_j}\dap^n$$ 
yields (using $\partial(A\times B)\subseteq(\partial A\times B)\cup(A\times\partial B)$ and $\partial U(\gotp)\subseteq\partial U_\gotp$)
\begin{equation} \label{e:brdprd} 
\partial U_{\sigma_j}\subseteq W_j:=\bigcup_{\gotp\in S_j}\left(\partial\,U_\gotp\times\prod_{\gotq\in S_j-\{\gotp\}}U(\gotq)\times\prod_{\gotq\notin S_j}\hat D_\gotq^n\right)\,. 
\end{equation}
By hypothesis (PS1), each term in the union on the right-hand side of \eqref{e:brdprd} has measure $0$. Hence $\mu(W_j)=\mu(\partial\,U_{\sigma_j})=0$ and Lemma \ref{l:brtt} shows that 
$d(U_{\sigma_j})=\mu(U_{\sigma_j})$ holds for every $j$.

For every $\gotp$, let $Z_\gotp=\hpi_{\gotp^\infty}^{-1}(U_\gotp)=U_\gotp\times\prod_{\gotq\neq\gotp}\da_\gotq^n\,.$ Then $U=\cap_\gotp Z_\gotp$ and we obtain
$$U_{\sigma_j}-U\subseteq W_j\cup(U_{\sigma_j}-\bigcap_{\gotp\notin S_j}Z_\gotp)=W_j\cup\bigcup_{\gotp\notin S_j}(U_{\sigma_j}-Z_\gotp)\subseteq W_j\cup\bigcup_{\gotp\notin S_j}(\da^n-Z_\gotp).$$
Note that $D^n\cap(\da^n-Z_\gotp)=D^n-U_\gotp$. Moreover $W_j$ is closed and thus $\mu(W_j)=0$ implies $d(W_j)=0$. Hence
$$d^+(U_{\sigma_j}-U)\le d(W_j)+d^+\left(\bigcup_{\gotp\notin S_j}\big(D^n-U_\gotp\big)\right)=d^+\left(\bigcup_{\gotp\notin S_j}\big(D^n-U_\gotp\big)\right)$$
shows that \eqref{e:ps} implies $\lim d^+(U_{\sigma_j}-U)=0$. Therefore we can apply Theorem \ref{t:mt} to $U$, getting
$$d(X)=d(D^n\cap U)=d(U)=\mu(\hat U)=\prod\mu(U(\gotp))=\prod\mu(U_\gotp)\,.$$
\end{prf}

\begin{cor} \label{c:psopelr} For every $\gotp\in\calp(D)$ let $U_\gotp\subseteq\dap^n$ be compact open and take $X:=D^n\cap\prod U_\gotp$. If \eqref{e:ps} holds and $\prod\mu_\gotp(U_\gotp)>0$ then $X$ is openly Eulerian, with $\xa=\prod_\gotp U_\gotp$. \end{cor}

\begin{prf} If $U_\gotp$ is compact open then $\partial U_\gotp=\emptyset$ and so (PS1) is automatically true. Now just apply Propositions \ref{p:ps} and \ref{p:d>mc}. \end{prf}

\begin{rmk} Consider a partition $\calp(D)=\bigsqcup_{i\in\N}T_i$ and let $\da_i$ be as in \eqref{e:diTi}. Proposition \ref{p:ps} is still valid (with the same proof) replacing the sets 
$U_\gotp$ with $U_i\subseteq\da_i^n$. In this situation, Corollary \ref{c:psopelr} yields a set of openly Eulerian type. \end{rmk}

\subsubsection{Ekedahl's theorem} The following is a rewording of \cite[Theorem 1.2]{ekd}.

\begin{thm}[Ekedahl] \label{t:ekd} Let $N$ be a positive integer and $\Sigma$ any subset of $(\Z/N\Z)^n$. Also, let $Z$ be a closed subscheme of the affine space $\A_\Z^n$ and put
$$Y:=\hpi_N^{-1}(\Sigma)\cap\bigcap_{p\nmid N}\hpi_p^{-1}(\F_p^n-Z(\F_p))\,.$$
Then the asymptotic density of the set $X:=Y\cap\Z^n$ exists and it is equal to $\mu(Y)$. \end{thm}

Ekedahl's original proof is not easy to follow; the simplest approach to Theorem \ref{t:ekd} and its generalization to global Dedekind domains is probably to use the Poonen-Stoll strategy in the following way. 

Let $U_p\subseteq\Z_p$ consist of those points whose reduction modulo $p$ is not in $Z(\F_p)$. Then each $U_p$ is compact open and $Y$ can be written as a product $U_N\times\prod_{p\nmid N}U_p$, where $U_N$ is a compact open subset of $\prod_{p|N}\Z_p$\,. With these $U_p$ and $d$ the asymptotic density, the validity of \eqref{e:ps}, with a precise estimate on the error term, is proved in \cite[Theorem 3.3]{bh} (see also \cite[Theorem 18]{bsw} for a similar result in the setting of global fields, from which one can deduce a version of Ekedahl's theorem with $\Z$ replaced by a global Dedekind domain $D$).\\

It is clear from the above discussion and Corollary \ref{c:psopelr} that the set $X$ in Theorem \ref{t:ekd} is of openly Eulerian type.

\begin{rmk} As remarked in \cite[Example 3]{cs} (and is implicit in \cite[Lemma 2]{ps1}), Ekedahl's theorem immediately yields the density of coprime pairs in $\Z^2$. With a little more work, the 
same method applies to check that sets of tuples with a fixed greatest common divisor in a global Dedekind domain have density equal to the measure computed in \eqref{e:prbbgcd}. \end{rmk}

\subsection{Some questions} Let $D$ be a global Dedekind domain.

When a set $X\subseteq D^n$ is defined by local conditions (e.g., by fixing $U_\gotp\subseteq\dap^n$ for every $\gotp$ and putting $X:=D^n\cap\prod U_\gotp$, as in Proposition \ref{p:ps}), one often expects it has (asymptotic) density computable by a product of local terms. If the local conditions correspond to open sets, then Proposition \ref{p:d>mc} can be used to deduce that $X$ is openly Eulerian. We provide a typical example.

\begin{prop} \label{p:dl2} Let $D$ be either the ring of integers of a number field or the ring of $S$-integers of a global function field, with $S$ any finite set of places. Let $f\in D[x]$ be a separable polynomial and $k\ge2$ an integer. In the number field case, assume also $\deg(f)\le k/[F:\Q]$. Then the set 
$$X_{f,k}:=\{a\in D\mid f(a)\text{ \em is }k\text{-\em free}\}$$
is openly Eulerian. \end{prop}

\begin{prf} For every $\gotp\in\calp(D)$, let $U_\gotp:=\{a\in\dap\mid f(a)\notin\gotp^k\}$ and $V_\gotp$ the set of zeroes of $f$ in $D/\gotp^k$. Then $X=D\cap\prod U_\gotp$; moreover, each $U_\gotp$ is compact open, because it is the lift to $\dap$ of the complement $V_\gotp$, and the same observation yields
$$\mu_\gotp(U_\gotp)=1-\frac{|V_\gotp|}{\|\gotp\|^k}\;.$$ 
By the ideas of Proposition \ref{p:ps}, it is proved in \cite[Theorem 3.2 and Corollary 3.4]{sch} that the equality 
\begin{equation} \label{e:dl2} \dsas(X_{f,k})=\prod\mu_\gotp(U_\gotp) \end{equation}
holds for an asymptotic density satisfying conditions (Dn1)-(Dn7). Finally, the product on the right of \eqref{e:dl2} is positive because $k\ge2$ and $|V_\gotp|<\deg(f)$ for almost every $\gotp$, by \cite[Lemma 3.1]{sch}. One concludes by Proposition \ref{p:d>mc}. \end{prf}

More generally, for $D$ any global Dedekind domain and $k\ge2$, take $f\in D[x_1,...,x_n]$ and let $X_{f,k}$ consist of those $a$ such that $f(a)$ is $k$-free. Then it is conjectured 
\begin{equation} \label{e:plnmsqrfr} \dsas(X_{f,k})\stackrel{?}{=}\prod_\gotp\left(1-\frac{c_\gotp}{\|\gotp\|^{kn}}\right)\,, \end{equation}
where $c_\gotp$ is the number of zeroes of $f$ in $(D/\gotp^k)^n$. By the work of many mathematicians (especially for $k=2$), it is known that \eqref{e:plnmsqrfr} holds in a number of 
cases (see for example \cite{hoo}, \cite{rams}, \cite{grn}, \cite{poon}, \cite{bh}, \cite{bsw}, \cite{bsw2}), but its general validity is still an open problem. On the other hand, the same reasoning as in Proposition \ref{p:dl2} yields the following result.

\begin{thm} Let $D$, $f$, $k$ and $X_{f,k}$ be as above. If \eqref{e:plnmsqrfr} holds and the product on the right is non-zero, the set $X_{f,k}$ is openly Eulerian. \end{thm}

\noindent Thus the statement that $X_{f,k}$ is openly Eulerian can be seen as a weak form of conjecture \eqref{e:plnmsqrfr}.\\

One can be a bit more daring and ask: \begin{itemize}
\item[({\bf Q1})] if $Y\subseteq D^m$ is openly Eulerian and $f\colon D^n\rightarrow D^m$ is a polynomial map, is it true that the inverse image $f^{-1}(Y)$ is also openly Eulerian? \end{itemize}
The case of $X_{f,k}$ discussed above arises when the map $f$ is given by a single polynomial and $Y$ is the set of $k$-free elements. Taking as $Y$ a set of tuples with a fixed greatest common divisor, one obtains more examples where a positive answer to ({\bf Q1}) can be obtained joining Proposition \ref{p:d>mc} with known results (like \cite[Lemma 2]{ps1}, \cite[Theorem 3.2]{poon}, \cite[Theorem 3.1]{butt}, \cite[Theorem 4.1]{chs}).

Ekedahl's theorem suggests another set of potential candidates for being openly Eulerian, namely sets of points whose reduction modulo $\gotp^k$ lies out of $Z(D/\gotp^k)$ for every 
$\gotp\in\calp(D)$, with $Z$ a given closed subscheme of $\A^n_D$\,. In this case, a proof of Eulerianity could be interpreted as the weak version of a positive answer to the density question raised in \cite[\S1, Remark i)]{ekd}.\\

In the examples above, we can obtain Eulerianity by knowledge of the density. It becomes natural to ask if the opposite approach can work, that is, if a converse of Proposition \ref{p:d>mc} is true. \begin{itemize}
\item[({\bf Q2})] If $X$ is almost closed and of openly Eulerian type, is it true that it has upper (asymptotic) density $\mu(\xa)$ ? Or, more optimistically, that $\mu(\xa)$ is its density?
\end{itemize}

The condition that $X$ is almost closed is needed because any dense set is openly Eulerian, but, as we discussed in \S\ref{sss:rlclsd}, its density could be everything.

Some reason of hope for the truth of ({\bf Q2}) might come from Proposition \ref{p:ps}: if $X$ is openly Eulerian and almost closed then $X\cap\dap^n$ must be ``quite big'' and this can 
support the intuition that the sets $D^n-X(\gotp)$ are small enough to make \eqref{e:ps} true. However, the evidence we know of appears too scarce to seriously conjecture a positive answer 
for question ({\bf Q2}).

\end{document}